\documentclass[article,onefignum,onetabnum]{siamart171218}
\usepackage{graphicx,a4,epsf,subcaption}
\usepackage[utf8]{inputenc}
\usepackage[english]{babel}
\usepackage{cite}
\usepackage{graphicx}
\usepackage{epstopdf}
\usepackage{lipsum}
\usepackage{mathtools}
\usepackage{tikz}
\usepackage{pgfplots}
\usepackage[utf8]{inputenc}
\usepackage{graphicx,amsmath,amsfonts,color,a4,pifont}
\usepackage{latexsym,amssymb,epsf}
\usepackage[font=scriptsize, labelfont={sf,bf}]{caption}
\usepackage{ifthen}
\usepackage{float} 
\usepackage{multirow}

\graphicspath{{./plots/}}

\newcommand{\mat}{\left[ \begin{array}{c}  }
\newcommand{\rix}{\end{array} \right] }

\newcommand{\C} {\mathbb{C}}

\newcommand{\eps} {\varepsilon}
\newcommand{\Om}{\Omega}

\newcommand{\lsim}{\lesssim}

\newcommand{\U}{\mathcal{U}}

\newcommand{\V}{\mathcal{V}}

\newcommand{\A}{\mathbf{A}}

\newcommand{\B}{\mathbf{B}}
\newcommand{\I}{\mathbf{I}}
\newcommand{\M}{\mathbf{M}}
\newcommand{\N}{\mathbf{N}}
\renewcommand{\S}{\mathbf{S}}
\newcommand{\D}{\mathbf{D}}
\renewcommand{\P}{\mathbf{P}}
\newcommand{\R}{\mathbf{R}}
\newcommand{\J}{\mathbf{J}}
\renewcommand{\u}{\mathbf{u}}
\newcommand{\vc}{\mathbf{v}}
\newcommand{\f}{\mathbf{f}}
\newcommand{\Uc}{\mathbf{U}}

\newcommand{\g}{\mathbf{g}}
\newcommand{\x}{\mathbf{x}}
\newcommand{\Q}{\mathbf{Q}}
\newcommand{\y}{\mathbf{y}}
\newcommand{\w}{\mathbf{w}}

\newcommand{\Hh}{\mathbf{D}}

\newcommand{\Cnn}{\mathbb{C}^{n\times n}}

\newcommand{\nsp}{{\mathcal N}}
\newcommand{\ran}{{\mathcal R}}

\newcommand{\beqo}{\begin{eqnarray*}}
\newcommand{\beq}{\begin{eqnarray}}
\newcommand{\eeqo}{\end{eqnarray*}}
\newcommand{\eeq}{\end{eqnarray}}

\author{Luis Garc\'{i}a Ramos, Reinhard Nabben}

\title{A two-level shifted Laplace Preconditioner for Helmholtz Problems: Field-of-values analysis and wavenumber-independent convergence}

\begin{document}
\maketitle

\begin{abstract} One of the main tools for solving linear systems arising from
the discretization of the Helmholtz equation is the shifted Laplace
preconditioner, which results from the discretization of a perturbed Helmholtz
problem $-\Delta u - (k^2 + i \varepsilon )u = f$ where $0 \neq \varepsilon \in \mathbb{R}$
is an absorption parameter. In this work we revisit the idea of combining the
shifted Laplace preconditioner with two-level deflation and apply it to Helmholtz
problems
discretized with linear finite elements. We use the convergence theory of GMRES
based on the field of values to prove that GMRES
applied to the two-level preconditioned system with a shift parameter $\eps
\sim
k^2$ converges in a number of iterations independent of the wavenumber $k$, 
provided that the coarse mesh size $H$ satisfies a condition of the form
$Hk^{2} \leq C$ for some constant $C$ depending on the domain but independent
of the wavenumber $k$. This behaviour is sharply different to the standalone
shifted Laplacian, for
which wavenumber-independent GMRES convergence has been established only under
the condition that $\varepsilon \sim k$ by [M.J. Gander, I.G. Graham and E.A. Spence,
Numer. Math., 131 (2015),
567-614]. Finally, we present numerical evidence that
wavenumber-independent convergence of GMRES also holds for pollution-free
meshes, where the coarse mesh size  satisfies $Hk^{3/2} \leq C
$, and inexact coarse grid solves.
\end{abstract}

\maketitle

\section{Introduction}

In this work we study the solution of linear systems of equations
arising
from
the discretization of the Helmholtz equation. 
We concentrate here on the interior
Helmholtz problem with impedance boundary conditions, which for a domain
$\Omega \subset \mathbb{R}^d$ with boundary $\Gamma$, $k \in \mathbb{R}$, $f_1
\in L^{2}(\Omega)$ and $f_2 \in L^{2}(\Gamma)$ is defined as:
\begin{equation}\label{eq:helmholtz}
\left\{
  \begin{array}{rc}
   -\Delta u -k^2u &= f_1  \,\, \text{ in $\Omega$,} \\
  \partial_n u -iku &= f_2 \,\, \text{ in $\Gamma$}.
  \end{array}
\right.
\end{equation}

The solution of these systems of equations is one of the main computational
bottlenecks for solving inverse
problems in various applications, e.g., in exploration
geophysics and medical imaging. The standard
variational formulation of 
\eqref{eq:helmholtz} and its corresponding Galerkin discretization with
P1 finite
elements on a simplicial mesh $T_h$ of the
domain $\Omega$ leads to a linear system
\begin{equation}\label{eq:ls}
	\A\u=\f,
\end{equation}
where $\A \in \C^{N_h \times N_h}$, and $ \u,\f \in \C^{N_h \times N_h}$. Due
to
the
oscillatory character of the solutions, in order to obtain an accurate
approximation the number of gridpoints in one dimension should be at least
proportional to $k$, leading to linear systems of size $N_h  \sim k^d$ where
$d$ is the spatial dimension. However, the Galerkin solutions are affected by 
the pollution effect \cite{Babuska:2000cf,Ihlenburg:1998hm} and this rule is
not sufficient to mantain accuracy when using discretizations with low-order
finite elements for large wavenumbers. In this case the number of
points in one-dimension should be chosen proportional to $k^{3/2}$ leading to
very large linear systems of size  $N_h \sim k^{3d/2}$. Moreover, the matrix
$\A$ is
non-Hermitian and indefinite, making the efficient solution of \eqref{eq:ls}
with standard iterative techniques a huge challenge; for a survey see
\cite{Ernst:2012km}.

The development of fast solvers for the Helmholtz equation has been an
active research area over the last decades. Notable works include the
wave-ray method \cite{Brandt:1997uda,Livshits:2006eo}, methods based on domain
decomposition,
and sweeping-type preconditioners
\cite{Engquist:2011hi,Engquist:2011jv,ZepedaNunez:2015tj}
, see, e.g., the survey papers
\cite{Erlangga:2007gd,Ernst:2012km,Gander:2019kr} for more references.

A
very fruitful idea introduced in
the landmark paper 
\cite{Erlangga:2006fp}  is to
precondition \eqref{eq:ls} with the discretization of a Helmholtz problem
with absorption (or shifted Laplace problem) of the form
\begin{equation}\label{eq:sslaplace}
\left\{
  \begin{array}{rc}
   -\Delta u -(k^2 + i \eps )u &= f_1  \,\, \text{ in $\Omega$,} \\
  \partial_n u -iku &= f_2 \,\, \text{ in $\Gamma$},
  \end{array}
\right.
\end{equation}
where $\eps$ is a real positive parameter. After discretization of
\eqref{eq:sslaplace}
one obtains a
linear
system with coefficient matrix
$\A_{\eps}$. The complex shift $i\eps$ in \eqref{eq:sslaplace}
reduces the oscillations in the solutions
and allows the preconditioner $\A_{\eps}$ to be inverted with fast methods,
e.g. multigrid or domain decomposition. This preconditioner is known as the
complex shifted Laplacian (CSL) or shifted Laplace preconditioner,  and is 
a 
building block for several state-of-the-art solvers for the
Helmholtz
equation, see, e.g., the book \cite{Lahaye:2017wi} for a recent overview of
extensions and industrial applications of Helmholtz
solvers based on the shifted Laplacian. The resulting
preconditioned system
 \[\A\A_{\eps}^{-1}\g =\f, \,\, \u = \A_{\eps}^{-1} \g\]
can be solved more efficiently than the original system with Krylov subspace
methods for non-Hermitian
systems. Here we only consider the solution of the linear system with the GMRES
method. In practice,  the preconditioner is inverted approximately with a fast
method; denoting this approximation by $\D_{\eps}^{-1} \approx \A_{\eps}^{-1}$,
the linear system to be analyzed is
 \begin{equation}\label{eq:lsprac}
 \A\D_{\eps}^{-1}\g =\f, \,\,\u =   \D_{\eps}^{-1} \g.\end{equation}
 
Naturally
there is a tradeoff in the choice of the shift $\eps$, since a small
shift leads to faster convergence of the Krylov solver but only a large enough
shift will allow the (approximate) inversion of $\A_{\eps}$ with  fast methods.
To date, the most rigorous analysis on
how to choose the shift has appeared in \cite{Gander:2015jf}, and the
series of papers \cite{Cocquet:2017jw,Graham:2017eg,Graham:2018vf}. The authors
of \cite{Gander:2015jf} propose to separate the analysis of the
shifted Laplacian in two questions:
\begin{enumerate}
\item Assuming that $\A_\eps$ is inverted exactly, determine conditions on
$\eps$ for $\A_{\eps}$ to be a good preconditioner for $\A$.
\item Determine conditions on $\eps$ for $\D_{\eps}^{-1}$ to be a good
approximation for $\A_{\eps}^{-1}$, in particular when using multigrid or
domain decomposition methods.
\end{enumerate}
More rigorously, one can use the identity
\begin{align*}
\I - \A\D_{\eps}^{-1} = \I- \A_{\eps} \D_{\eps}^{-1} + \A_{\eps}
\D_{\eps}^{-1}(\I-\A\A_{\eps}^{-1}),
\end{align*}
to show that if both $\|\I- \A_{\eps}\D_{\eps}^{-1}\|$ and $\|\I-
\A\A_{\eps}^{-1}\|$ are small (i.e., the quantitative version of statements (a)
and (b)) then GMRES applied to \eqref{eq:lsprac}, is expected to converge fast.

Some of the early works on the shifted
Laplacian \cite{vanGijzen:2007ir,Erlangga:2006fp,Erlangga:2004bg}
focused
on
answering question (a) using the GMRES residual bounds based on the spectrum of
the matrix. In these papers the choice  $\eps = \beta k^2$ with $\beta \in
[0.5,1]$ was recommended, based on an analysis of the spectrum of the 
preconditioned (continuous) Helmholtz operator with Dirichlet boundary
conditions in 1D in
\cite{Erlangga:2007gd} and the spectrum of more general preconditioned discrete
Helmholtz operators in \cite{vanGijzen:2007ir}. Likewise, a combination of
questions (a) and (b)  was investigated in
\cite{Cools:2013hm}
using Local Fourier
Analysis, leading to an expression for the near-optimal value
 (for guaranteed  multigrid V-cycle convergence and a near-optimal minimum
number
of
GMRES
iterations) depending on the wavenumber and gridsize. For more references on
how to choose the shift in the CSL preconditioner
see \cite[Section 1.1]{Gander:2015jf}.

The
analysis of GMRES convergence based on the distribution of the eigenvalues of
the preconditioned matrix rests on the assumption
that the condition number of the matrix of eigenvectors of the preconditioned
matrix is small. However, this factor is hard to
estimate and for some problems it can be so large that the bound may not be
informative at all. Moreover, it is known that the convergence of
GMRES applied to a non-normal linear system cannot be predicted only by the
spectrum of the matrix \cite{Greenbaum:1996cp,Liesen:2012tt}. The authors of \cite{Gander:2015jf} use instead the convergence theory of GMRES
based on the field of values (which we summarize in section 3.) to study
question
(a)
above.
Their
main
result
shows
that under some natural assumptions on the geometry of the domain the condition $\eps
\lsim
k$ (i.e., $\eps \leq Ck$ with a small enough constant $C$)
is
sufficient
for
the
field
of
values
of
$\A\A_{\eps}^{-1}$ to be bounded away from the origin as $k \to
\infty$, and 
therefore
under this condition the number of
iterations of GMRES applied to the preconditioned system remains constant as
$k$ is increased.

Question (b) has been studied in \cite{Cocquet:2017jw} for
general 2D and 3D problems in the context of the multigrid method.
There it is shown that choosing
a shift $\eps \sim k^2$ is a necessary and sufficient condition for the
convergence of multigrid with a fixed number of (weighted Jacobi) smoothing
steps applied to a
linear system with coefficient matrix $\A_{\eps}$. Regarding domain
decomposition methods and question (b), the work \cite{Graham:2017eg}
investigates the requirements for an additive Schwarz preconditioner 
$\mathbf{D}_{\eps}^{-1} \approx \A_{\eps}^{-1}$  (with Dirichlet transmission
conditions between subdomains and  coarse grid correction) to be
effective as a preconditioner for a problem with coefficient matrix
$\mathbf{A}_{\eps}$, i.e., to obtain fast convergence of GMRES applied to a
linear system with coefficient matrix $\mathbf{A}_{\eps}
\mathbf{D}_{\eps}^{-1}$.
There it is concluded that a sufficient condition is $\eps \lsim k^2$. The more recent
paper  \cite{Graham:2018vf} introduces an additive Schwarz method with local
impedance boundary conditions as transmission conditions between subdomains and
shows that it is possible to obtain wavenumber independent convergence also
under the condition $\eps \sim k^{1+\beta}$ for $\beta$ arbitrarily small.
Together, these recent results imply that there exists a rigorously quantified
gap between the condition for having a good preconditioner (which requires a small shift $\eps
\lsim k$) and the requirement for being able to (approximately) invert $\A_{\eps}$ (for which $\eps
\sim k^2$ or at least $\eps \sim k^{1+\beta}$ with $\beta >0$ arbitrarily small is
necesary), thus motivating the need for more advanced preconditioning techniques.

In this paper we revisit the combination of the shifted Laplacian
with two-level deflation, which was introduced in
\cite{Sheikh:2013iq,Sheikh:2016eg}
based on previous work in \cite{Erlangga:2008va}, see also
\cite{GarciaRamos:2018gy} for an analysis of the method proposed in
\cite{Erlangga:2008va}. We extend the
simplified variant
of
two-level deflation from \cite{Sheikh:2013iq,Sheikh:2016eg}
to Helmholtz problems discretized with finite
elements, and we contribute to the line of analysis proposed in
\cite{Gander:2015jf} by studying the analogous of question (a) for the
two-level
shifted Laplacian combined with deflation. We use the framework for the analysis of two-level preconditioners from \cite{Hannukainen:2011tya}.

The method that we use combines the shifted Laplacian
with a 
projection-based
preconditioner which removes the components of
$\A\A_{\eps}^{-1}$
that cause the slow convergence of a Krylov subspace method applied to the
linear
sytem. 
Algebraically
this
idea
can
be
motivated
as
follows. Let $\Q \in \Cnn$ be a projection operator, i.e., a matrix such that
$\Q^{2}= \Q$. If $\ran(\Q)$ is the range of $\Q$ and $\nsp(\Q)$ its nullspace,
the direct sum decomposition $\mathbb{C}^n = \ran(\Q) \oplus \nsp(\Q)$ holds
and $\A^{-1}$
can be split as
\[ \A^{-1} = \A^{-1}(\I-\Q) +\A^{-1}\Q, \]
note that $(\I-\Q)$ is also a projection. Therefore, a sensible approximation
for the inverse  of $\A$ is
\begin{equation}\label{eq:projprec}
\B = \A_{\eps}^{-1}(\I-\Q) +\A^{-1}\Q.\end{equation}
The crucial point is that the projection can be chosen so that the term
$\A^{-1}\Q$ on the right can be computed cheaply  by solving a smaller linear
system. If $m <n$ and $\mathcal{U}$
is an  subspace of $\C^n$ spanned by the columns of a full rank matrix $\Uc \in
\C^{n \times m}$
such that $\Uc^\ast \A \Uc$ is nonsingular (where the superscript $*$ denotes
the conjugate transpose), the projection $\Q$ with $\ran(\Q) = \A \mathcal{U} $
and $\nsp( \Q) = \mathcal{U}^{\perp}$ (the orthogonal complement of $\U$ in the
Euclidean inner product) has the form 
\[\Q = \A \Uc (\Uc^*\A\Uc)^{-1}\Uc^*,\]
which gives
\begin{equation}\label{eq:projectionprec}
\B = \A_{\eps}^{-1}(\I-\A\Uc (\Uc^*\A\Uc)^{-1}\Uc^*)
+\Uc(\Uc^*\A\Uc)^{-1}\Uc^*
\end{equation}
and the term $\Uc(\Uc^*\A\Uc)^{-1}\Uc^*$ can be computed by inverting a smaller
system with coefficient matrix $\Uc^*\A\Uc$. It follows easily using this
projection representation that if $\B$ is used as a preconditioner
for
$\A$
then $\A\B$ equals the identity when restricted to the subspace $\U$,
so the spectrum of the preconditioned matrix $\A\B$ contains one as an
eigenvalue with multiplicity (at least) $m= \mathrm{dim}(\U)$. Moreover, if the
subspace $\U$ contains the solution $\u$ we have $\f = \A \u \in \A\U$, and
GMRES applied to the preconditioned linear system $\A\B \u = \f$ with a
zero
initial guess will converge in one step.  Projection-based preconditioners of
the form \eqref{eq:projectionprec} are related to classical deflation methods
in which a projection operator is used to remove near-singular eigenspaces
responsible for slowing down the convergence of a Krylov subspace iteration,
the main difference being that in classical deflation the resulting deflated
linear system is singular, i.e., the eigenvalues are shifted to zero, not to
one. For more on the connection between two-level methods, projections and
deflation see \cite{GarciaRamos:2020fu}.
 
An important class of methods that  lead to preconditioners of the
form
\eqref{eq:projectionprec}
are two-level (or two-grid) methods, which are the  basis of the
multigrid
method \cite{Trottenberg:2001uw}. In
\cite{Sheikh:2013iq,Sheikh:2016eg},
the
authors consider a finite difference discretization of the problem on a grid
$G_h$ and a coarse grid $G_H \subset G _h$, and choose the subspace $\U$ as
the span of the columns of (the matrix representation of) the two-grid
prolongation operator $\P$. Here we analyze an extension of this preconditioner
to the finite element setting, where the mesh $T_h$ is coarsened by
choosing a mesh $T_H$ such that the elements in $T_H$ are unions of
elements of $T_h$. If $\V_h, \V_H$ are spaces of P1 finite elements  to $T_h$
and $T_H$ respectively, we then have $\V_H
\subset \V_h$ and the prolongation operator is defined trivially by  this
inclusion.  The resulting preconditioner is (see section 4. for    
more details)
\[ \B_{\eps} = \A_{\eps}^{-1}(\I-\A \P \A_H^{-1}\P^*) + \P \A_H^{-1}\P^*,
\]
leading to the preconditioned system \begin{equation*}
	\A\B_{\eps}\g=\f.
\end{equation*}

The main result in our paper is Theorem \ref{thm:main}, where we prove that it
is possible to close the gap between the requirements for $\eps$, i.e., we
show that
wavenumber-independent GMRES convergence can be obtained with the two-level
shifted Laplacian
even
in the
case of a large shift $\eps \sim k^2$, provided that the coarse grid size
satisfies a condition of the form $Hk^{2} <C$ for some constant $C>0$ depending
only on the domain $\Omega$ (but independent of the wavenumber $k$). 
Note that in this theorem we are assuming that the shifted Laplacian  and the
coarse grid system are inverted exactly.
  
This result also confirms what has been previously observed in
the spectral analysis of a 1D model problem in \cite{Sheikh:2013iq, Lahaye:2017wi}
where
it has been
shown (using
Fourier
analysis on a one-dimensional Helmholtz problem with Dirichlet boundary
conditions) that when the shifted Laplacian is combined with two-level
deflation
 with a complex shift
$\eps = \beta k^2$ it is possible to increase $\beta$ without greatly affecting
the spectrum of the preconditioned matrix.

We
prove
Theorem \ref{thm:main} for weighted GMRES in the inner product induced by the
inverse of the domain mass matrix,
 and show that this norm is a natural norm to measure the
residuals
of
the
preconditioned system since it corresponds to the dual $L^2$ norm when
$\C^{N_h}$ is identified with the space of coordinates of $\V_h'$ (Proposition
\ref{thm:dual}, (c)). Fortunately, for a sequence of quasi-uniform meshes one
can use a scaling argument and norm equivalences to show that the result also
holds for GMRES in the Euclidean inner product (Corollary
\ref{thm:corollarynorm}, part (a)).

This paper is organized as follows: In section 2. we review some basic results
on the variational formulation of Helmholtz  problems, focusing on the
conditions for existence and uniqueness of solutions to these problems and
their stability. In sections 3. and 4. we introduce the finite element
formulation of the Helmholtz and shifted Laplace problems and the convergence
theory of GMRES based on the field of values. In section 5. the two-grid
preconditioner is introduced and we prove our main result. Finally, in section
6.
we present some numerical experiments to illustrate our results.
 
\section{Preliminaries: A recap of the variational formulation and finite
element approximation of Helmholtz problems}

In this section we review some basic results on the variational
formulation
of
Helmholtz  problems, focusing on the conditions for existence and uniqueness of
solutions to these problems and their stability. In the second part we discuss
the finite element approximation of Helmholtz problems with the Galerkin
method. We refer the reader to \cite{Spence:2015jl} for a very good
introduction
to
the
variational formulation of Helmholtz problems. To simplify the notation, we
will write
$a
\lsim b$ to denote that there exists a constant $C$ such that $a \leq C b$
independent of the  parameters on which $a$ and $b$ may depend. Moreover, we
write $a \sim b$ when $a \lsim b$ and $b \lsim a$. 

Given a complex inner product space $\V$ we denote its sesquilinear inner
product by $(\cdot,\cdot)_{\mathcal{V}}$. The antidual space (of
continuous conjugate-linear functionals from $\V$ to $\C$)
is denoted by $\V^{'}$.
The duality pairing
$\langle \cdot,\cdot \rangle_{\V' \times \V}: \V' \times \V \to \C$  is defined
for $f \in \V'$, $v \in \V$ as
\[\langle f,v \rangle_{\V' \times \V}  =f(v).\]
The dual norm in the space $\V^{'}$ is defined by
\[\|f\|_{\V^{'}} =  \sup_{0 \neq v \in \V} \frac{|\langle f,v\rangle _{\V'
\times \V}|}{\|v\|_{\V}}.\]
We will drop the subscripts $\V',\V$ when this introduces no ambiguities, and
write only $\langle \cdot, \cdot \rangle$ and $\|\cdot\|$. Let
$\Omega
\subset
\mathbb{R}^d$ be  a convex  polyhedron in $\mathbb{R}^d$ (where $d=1,2,3$) with
boundary $\Gamma$.  Recall that the Sobolev space $L^2(\Omega)$ of square
integrable functions
is equipped
with
the
inner
product
\[(u,v)_{L^2(\Omega)}= \left(\int_{\Om}u\overline{v} \right)^{1/2}.\]  
The higher order
Sobolev spaces $H^{m}(\Omega)$ consist of functions $u \in L^2(\Omega)$ that
have weak derivatives $\partial^{\alpha} u$ in $L^{2}(\Omega)$ for all
multi-indices
$\alpha$ with $|\alpha| \leq m$. The standard inner product in $H^{m}(\Omega)$
is given by 
 \[(u,v)_{m}  = \sum_{|\alpha| \leq m} (\partial^{\alpha} u, \partial^{\alpha}
v)_{L^2(\Omega)},\]
 and the induced norm is denoted by $\|\cdot\|_{m,\Omega}$. In the space
$H^{m}(\Omega)$
we also introduce the seminorm
 \[|u|_{m} =  \sum_{|\alpha| = m} \|\partial^{\alpha} u\|_{L^2(\Omega)}.\] 
The variational formulation of the Helmholtz problem requires the use of a
special $k$-weighted inner product in the space 
 $H^{1}(\Omega)$. Given $k \in \mathbb{R}$, we define
the
Helmholtz
energy
inner
product on  $H^{1}(\Omega)$ by
\[(u,v)_{k,1, \Om}= (\nabla u, \nabla v)_{L^2(\Om)}+k^2(u,v)_{L^2(\Om)},\]
for any $u,v \in  H^{1}(\Omega)$. The induced norm will be denoted by $\|\cdot
\|_{1,k,\Om}$. After multiplying each side of \eqref{eq:helmholtz} with a test
function $v \in
H^1(\Omega)$, integrating by parts and substituting the boundary condition, the
problem can be restated in the variational form
\begin{equation}\label{eq:helmholtzvar1}
\text{Find $u \in H^{1}(\Omega)$ such that } \, a(u,v)=\langle f, v \rangle
\text{ for all } v \in H^{1}(\Omega),
\end{equation}
where the sesquilinear form $a: H^{1}(\Omega) \times H^{1}(\Omega) \to \C $ and
the antilinear functional $f: L^{2}(\Omega) \to \C$ are given by
\begin{align} \label{eq:helmform}
	a(u,v)&= \int_{\Omega} \nabla u \overline{\nabla v} - k^2 \int_{\Omega} u
\overline{v} - ik \int_{\Gamma}u\overline{v},\\
	\langle f,v \rangle &= \int_{\Omega} f_1 \overline{v} + \int_{\Gamma}
f_2\overline{v}.	
\end{align}

The following lemma summarizes some properties of the sesquilinear form of the
Helmholtz problem.

\begin{lemma} Let $a$ be the sesquilinear form~\eqref{eq:helmform}
 of the Helmholtz problem. The following properties hold:
\begin{itemize} 
\item[(a)] \cite[Lemma 8.1.6]{Melenk:1998tb},~\cite[p.118]{Spence:2015jl} The
form $a$ is continuous with continuity constant $C_c$ independent of $k$, i.e.,
there exists $C_c$ such that for all $u,v \in H^1(\Omega), k \in \mathbb{R}$:
\[|a(u,v)| \leq C_c\|u\|_{1,k,\Omega} \|v\|_{1,k,\Omega}\]
\item[(b)] The form $a$ satisfies the G\aa rding inequality (as an equality)
	\begin{equation}
		\|u\|_{1,k,\Omega}^2 = \Re a(u,u)+2k^2\|u\|_{L^2(\Omega)}^2.
	\end{equation}
\end{itemize}
\end{lemma}

It can be shown \cite[Lemma 6.17]{Spence:2015jl}
that associated to the
sesquilinear form $a$ there exists a bounded operator $\mathcal{A}:
H^{1}(\Omega)\to H^{1}(\Omega)'$ such that
\begin{equation}\label{eq:operatorA}
a(u,v) = \langle \mathcal{A} u , v \rangle, 
\end{equation}

using the operator $\mathcal{A}$, the variational problem can be rewritten as
\begin{equation}\label{eq:helmholtzvar1}
\text{Find $u \in H^{1}(\Omega)$ such that } \mathcal{A} u = f.
\end{equation}

The next theorem gives a stability estimate for the Helmholtz problem. For the
definition of the norm $\|\cdot \|_{1/2, \Gamma}$ see \cite[Section
6.2.4]{Hackbusch:1992fl}

\begin{theorem} \label{thm:stab}
Let $\Omega \subset \mathbb{R}^d$ be a bounded convex domain with boundary
$\Gamma$, where $d=2,3$. Given $k_0>0$, there exists a constant $C$ (depending
only on $\Omega$) such that for any $g \in L^2(\Omega), h \in L^2(\Gamma)$ and
$k > k_0$ the solution of the Helmholtz problem satisfies	
\begin{align}
		\|u\|_{1,k,\Omega} &\leq C(\|f_1\|_{L^2(\Omega)}+\|f_2\|_{L^2(\Gamma)}).
\end{align}
Moreover, if $u \in H^2(\Omega)$:
\begin{equation}
|u|_{H^2(\Omega)} \leq  C[(1+k)(\|f_1\|_{L^2(\Omega)} +\|f_2\|_{L^2(\Gamma)}
)+\|f_2\|_{1/2,\Gamma}].
\end{equation}
\end{theorem}
\begin{proof}
See \cite[Prop. 8.1.4]{Melenk:1998tb} for the case $d=2$ and
\cite[Prop.
3.4,
3.6]{Hetmaniuk:2007bx} for the case $d=3$.
\end{proof}

\subsection{Finite element approximation of Helmholtz problems}

In this section we recall the Galerkin formulation of the Helmholtz problem.
Let $\{T_h\}_{h >0}$ be a family of conforming simplicial meshes of
$\Omega$, where 
$h
=
\max_{K \in T_h}
\mathrm{diam}(K)$ denotes the mesh diameter. We assume that the family
$\{T_h\}_{h
>0}$ is shape-regular, i.e., \[ \sup_{h > 0 } \max_{K \in T_h}
\frac{\mathrm{diam}(K)}{\rho(K)} \leq C < \infty.\]

We let $\V_h$ be the space of P1 finite elements subordinate to $T_h$ spanned
by the standard nodal basis
$\Phi_h=(\phi_1,\ldots,\phi_{N_h})$. 
To simplify the notation, in what follows we omit the subscripts and denote the
$L^2$ inner product by $(\cdot, \cdot)$  and the corresponding norm by
$\|\cdot\|$.
Every element $u= \sum_{i=1}^{N_h} u_i \phi_i \in \V_h$ can be represented by
the vector of coordinates $\u = (u_1, \ldots, u_{N_h}) \in \C^{N_h} $, we write
this correspondence as
\[ u = \Phi_h \u .\]
Associated to $\Phi_h$ there exists a canonical basis of antilinear functionals
$\Phi^{'}_h=(\phi_1^{'}, \ldots, \phi_{N_h}^{'})$ for the space $\V_h^{'}$,
that satisfies
\[\langle \phi_{i}^{'}, \phi_{j}\rangle  = \phi_i^{'}(\phi_j)= \delta_{ij}, \,
\text{for } \, i,j,=1,  \ldots,N_h, \]
We write $f = \Phi_h' \f$ for the coordinate correspondence between $\C^{N_h}$
and $\V_h'$. In this notation, we have
\[ \u = (\langle \phi_1^{'},u \rangle, \ldots, \langle \phi_{N_h}^{'}, u
\rangle)^*,\]
here $^*$ denotes the conjugate transpose.  Recall that if $\Hh \in
\C^{N\times N}$ is a Hermitian positive definite (HPD) matrix, the inner
product $(\cdot, \cdot)_{\Hh}$  induced by $\Hh$ on $\C^N$ is defined as
 \[(\x, \y)_{\Hh} = \y^* \Hh \x,\]
 the corresponding inner product on $\mathbb{C}^n$ will be denoted by
$\|\cdot\|_{\Hh}$.

When $\C^{N_h}$ is identified with the coordinate space of $\V_h$ via
$\Phi_h$, the domain mass matrix $\M \in \C^{N_h \times \N_h}$ defined as
\begin{equation}\label{eq:mass}
	\M _{ij} = ( \phi_j, \phi_{i})_{L^2(\Omega)},  \,\, 1 \leq i, j \leq N_h.
\end{equation}
induces a norm in $\C^{N_h}$ corresponding to the $L^2$ norm in the space
$\V_h$.
This  implies that for all $\u, \vc \in \C^{N_h}$ and $u=\Phi_h \u,
v= \Phi_h \vc\in \V_h$:
\begin{align*} (\u,\vc)_\M &= (u,v)_{L^2(\Omega)}.
\end{align*}
The Riesz representation theorem implies that for every $f \in \V_h'$ there
exists a unique $u_{f} \in \V_h$ such that 
\[\int_{\Omega} u_f \overline{v}=\langle f,v \rangle,\]
for all $v \in \V_h$. The mapping $\tau: \V_h' \to \V_h$ defined by
$\tau(f)=u_f$ is called the Riesz map (with respect to the $L^2$ inner
product). The corresponding representation in $\C^{N_h}$ of the duality pairing
$\langle \cdot, \cdot \rangle$, the $L^2$ Riesz map and the dual norm is
explained in the next proposition, for completeness we include the proof from
chapter 6 of \cite{Malek:2014je}.

\begin{proposition} \label{thm:dual}
	The following statements hold:
\begin{enumerate}
 \item The duality pairing $\langle \cdot, \cdot  \rangle: \V_{h}^{'} \times
\V_h \to \C$ is represented by the Euclidean product in $\C^{N_h}$, that is,
for $\f, \u \in \C^{N_h}$, $u=\Phi_h \u $ and $f= \Phi^{'}_h\f$:
 \[   \langle f,u \rangle  = \u^*\f. \]

\item The matrix representation of the $L^2$ Riesz map $\tau: \V_h' \to  \V_h 
$
is

the inverse of the mass matrix $\M$:
  \begin{equation} \label{eq:invmass}\M^{-1}=\Phi_h \circ \tau \circ \Phi' \in \C^{N_h \times N_h}.
  \end{equation}
  
 \item The norm in  $\C^{N_h}$ corresponding to the dual norm in $\V_h'$ is the
norm induced by the inverse of the mass matrix $\M^{-1}$ from \eqref{eq:invmass}, that is, for $\f \in \C^{N_h}$ and $f= \phi^{'} \f$
we have:
  \[ \| \f \|_{\M^{-1}} = \sup_{\u \in \C^{N_h}} \frac{|\u^*\f| }{\|\u\|_\M}=
\sup_{u \in \V_h} \frac{|\langle f, u \rangle |}{\|u\|}. \]
\end{enumerate}
\end{proposition}
\begin{proof} For part (a), let $\u = (u_1,\ldots,\u_{N_h})$ and $\f =
(f_1,\ldots,\f_{N_h})$. We have
\begin{align*}
\langle f,u \rangle &= \langle \sum_{i=1}^{N_h} f_i \phi_i',
\sum_{j=1}^{N_h}
u_j \phi_j  \rangle \\
&= \sum_{i,j=1}^{N_h} f_i \overline{u}_j \langle \phi_i', \phi_j \rangle =
\sum_{i=1}^{N_h} f_i \overline{u}_i = \u^*\f.
\end{align*} 

For part (b), suppose that $\M_\tau \in \C^{N_h \times N_h}$ is the matrix
representation of the Riesz map $\tau: \V_h' \to \V_h$. With part (a) we
obtain
\begin{align*}
\u^*\f = \langle f,u \rangle &= (\tau f, u)=(\M_\tau \f, \u)_{\M} = \u^* \M
\M_\tau \f,  \text{ for all } \f,\u \in \C^{N_h},
\end{align*}
which implies $\M\M_\tau =\I$, so $\M_\tau=\M^{-1}$. For part (c), given $\f
\in
\C^{N_h}$
the
Cauchy-Schwarz
inequality in the inner product induced by $\M^{-1}$ implies that for every $
\g
\in
\C^{N_h}$:
\[ |(\f,\g)_{\M^{-1}}|  \leq \|\f\|_{\M^{-1}} \|\g\|_{\M^{-1}},  \]
with equality when $\g$ is a scalar multiple of $\f$. Therefore,
\begin{align*}
\|\f\|_{\M^{-1}} = \frac{|(\f,\f)_{\M^{-1}}|}{\|\f\|_{\M^{-1}}} &= \sup_{\g \in
\C^{N_h}} \frac{|(\f,\g)_{\M^{-1}}|}{\|\g\|_{\M^{-1}}}\\
& = \sup_{\g \in \C^{N_h}}\frac{|(\M^{-1}\g)^*\f|}{\|\g\|_{\M^{-1}}}\\
& = \sup_{\u \in \C^{N_h}}\frac{|(\M^{-1} \M \u)^*\f|}{\|\M \u \|_{\M ^{-1}}}
\\
&= \sup_{\u \in \C^{N_h}}\frac{|\u^*\f|}{\|\u \|_{\M}} =  \sup_{u \in
\V_h}\frac{|\langle f,u\rangle|}{\|u \|}.
\end{align*}
\end{proof}

The Galerkin problem in $\V_h$ takes the form 
\begin{equation}\label{eq:helmholtzgal}
\text{Find $u \in \V_h$ such that } \, a(u,v)=f(v) \text{ for all } v \in \V_h.
\end{equation}

Using the  operator $\mathcal{A}_h: \V_h \to \V_h'$ defined as in
\eqref{eq:operatorA} we see that the Galerkin problem is equivalent to finding
a
solution $u \in \V_h$ to the functional equation 
\[\mathcal{A}_h u =f|_{\V_h},\]
where the right hand side is the restriction of $f$ to $\V_h$. If $u=\Phi \u$
and $f = \Phi' \f$, we obtain the linear-algebraic formulation of the Galerkin
problem
 \begin{equation}\label{eq:helmholtzAu}
 \A \u = \f, \end{equation}

where the  matrix $\A \in \C^{N_{h}\times N_{h}}$ and the  vector $\f \in
\C^{N_{h}}$ are given by
\begin{align*}\label{eq:Au=f}\nonumber
 \A _{ij} &= a(\phi_j,\phi_{i}), \,\,\,  1 \leq i,j\leq N_h.
 \nonumber \\
  \f & = (\langle f,\phi_1\rangle , \ldots, \langle f,\phi_{N_h}\rangle )^T.
\end{align*}

In the case of Helmholtz problems with the sesquilinear form $a$ given by
\eqref{eq:helmform}, the matrix $\A$ from the linear system has the form
 \[\A = \S -k^2\M - ik\N,\]
where $\M$ is the mass matrix \eqref{eq:mass}, and $\S,\N$ are defined by 
\begin{align*}\label{eq:Au=f}\nonumber
 \S _{ij} &= (\nabla \phi_j, \nabla \phi_{i}), \,\,\,
  \N_{ij} = ( \phi_j, \phi_{i})_{L^2(\partial \Omega)}, \,\,\,  1 \leq i,j\leq
N_h.
\end{align*}

To finish this section, we discuss the approximability properties of the space
 $\V_h$. We assume that $\V_h$ is a space of  piecewise linear Lagrange
finite elements on a simplicial  mesh (triangular or tetrahedral, in 2D or 3D
respectively).  Under this assumption, the Scott-Zhang
interpolation operator  $\Pi_{SZ}:
H^1(\Omega) \to \V_h$ is well defined (see \cite[Section
1.6.2]{Ern:2013dd}).
Using the norm equivalence
\[\|u\|_{1,k,\Omega} \sim k \| u \|_{L^2(\Omega)} + \|\nabla
u\|_{L^2(\Omega)}\]
and standard interpolation estimates (see \cite[Lemma 1.130]{Ern:2013dd})
it
can be shown that the Scott-Zhang interpolation operator $\Pi_{SZ}$ has the
property that for all $w \in H^2(\Omega)$:
\begin{equation}\label{eq:szint1}
\|w-\Pi_{SZ} w\|_{1,k,\Omega} \lsim h \|w\|_{H^2(\Omega)} + hk
\|w\|_{H^1(\Omega)}
\end{equation}
 and for $w \in H^{1}(\Omega)$ 
\begin{equation}\label{eq:szint2}
\|w-\Pi_{SZ} w\|_{1,k,\Omega} \lsim (1+hk) \|w\|_{H^1(\Omega)}.
\end{equation}

\subsection{The shifted Laplace preconditioner and GMRES}
Given $\eps >0$, we consider the Helmholtz problem with absorption (or 
``shifted Laplace" problem)
\begin{equation}\label{eq:slaplace2}
\left\{
  \begin{array}{rc}
   -\Delta u -(k^2 + i \eps )u &= f_1  \,\, \text{ in $\Omega$,} \\
  \partial_n u -iku &= f_2 \,\, \text{ in $\Gamma$}.
  \end{array}
\right.
\end{equation}
with corresponding variational formulation
\begin{equation}\label{eq:helmholtzvar2}
\text{Find $u \in H^{1}(\Omega)$ such that } \, a_{\eps}(u,v)=\langle f, 
v \rangle
\text{
for all } v \in H^{1}(\Omega),
\end{equation}
with the sesquilinear form $a_\eps$ and the antilinear functional $f$ given by
\begin{align}
	a_{\eps}(u,v)&= \int_{\Omega} \nabla u \overline{\nabla v} - (k^2 + i
\eps)\int_{\Omega} u \overline{v}- ik \int_{\Gamma}u\overline{v} \\
	\langle f, v \rangle &= \int_{\Omega} f_1 \overline{v} + \int_{\Gamma}
f_2\overline{v}.
\end{align}
The next theorem summarizes the properties of the sesquilinear form $a_\eps$.
\begin{lemma} \cite[Lemma 3.1]{Gander:2015jf}  Let $a_\eps$ be the
sesquilinear form of the shifted Laplace problem \label{eq:slaplace1}. The
following properties hold: 
\begin{itemize}

\item[1.] The form $a_\eps$ is continuous, that is, if $0 < \eps \lsim k^2$
then given
$k_0>0$
there
exists
a constant
$C_c$
independent of $k, \eps$ such that  for all $k >k_0$, and $u,v \in
H^1(\Omega)$
\[|a_{\eps}(u,v)| \leq C_c\|u\|_{1,k,\Omega} \|v\|_{1,k,\Omega}.\]

\item[2.]  The  form $a_\eps$ is coercive, that is, if $0 < \eps \lsim
k^2$ there exists a constant $\alpha >0$
independent of $k,\eps$ such that for all $k > 0$ and  $u \in H^1(\Omega)$	
	\[|a_{\eps}(u,u)| \geq \alpha \frac{\eps}{k^2}\|u\|_{1,k,\Omega}^2.\]
	\end{itemize}
\end{lemma}

The matrix $\A_{\eps}$ corresponding to the discrete shifted Laplace problem
has the form
\[\A_\eps= \S -(k^2+ i \eps )\M - ik\N = \A - i\eps\M.\]

We will use in our analysis a bound for the GMRES residuals based on the field
of values. Recall that given  a Hermitian positive definite (HPD) matrix $\Hh
\in
\C^{N
\times N}$ and an arbitrary  $\mathbf{C} \in \C^{N \times N}$, the
field of values of $\mathbf{C}$ in the inner product induced by $\Hh$ is the
set
\[\mathcal{F}_{\Hh}(\mathbf{C})=\left\{\frac{( \mathbf{C} \x,\x )_{\Hh}}{(\x,
\x )_{\Hh}}:  \x \in \C^{N},  \x  \neq  \mathbf{0} \right\}. \]
Note that the spectrum of $\mathbf{C}$ is contained in
$\mathcal{F}_{\Hh}(\mathbf{C})$ for any HPD matrix $\Hh$. The
Toeplitz-Hausdorff theorem states that the field of values is a convex, compact
set \cite{Horn:1994tx}, hence the following quantity is well defined:
\[\nu_{\Hh}(\mathbf{C})= \min_{z \in \mathcal{F}_{\Hh}(\mathbf{C})}|z|.
\]  
The next theorem by Elman \cite{Elman:1982vx} shows
that
quantities
related
to
the
field of values can be used to bound the residuals of a minimum residual method
in an arbitrary inner product. 

\begin{theorem} Let $\mathbf{r}_0$ be the initial residual of the GMRES method
applied to the
linear system
\[\mathbf{C} \mathbf{u}=\mathbf{f}\]	
	 in the inner product induced by $\Hh$. The $n$-th residual
$\mathbf{r}_n$ satisfies:
\begin{equation}\label{eq:elmanbound}
\frac{\| \mathbf{r}_n \|_\Hh}{\| \mathbf{r}_0\|_\Hh} \leq \left(1-
\frac{\nu_{\Hh}(\mathbf{C})^{2}}{\|\mathbf{C}\|_{\Hh}^2}\right)^{n/2},
\end{equation}
\end{theorem}

The proof of the GMRES bound based on the field of values relies on the fact
that
the residual minimization problem solved by a GMRES iteration in step $j$ can
be restricted from the $j$-dimensional Krylov subspace to a one-dimensional
subspace, so in general one cannot expect the bound \eqref{eq:elmanbound} to be
sharp in the intermediate steps of an iteration. Nevertheless, for
(preconditioned) linear systems that result from finite element discretizations
of PDEs, one can estimate the quantities $\nu_\D(\mathbf{C})$ and
$\|\mathbf{C}\|$ using
properties of the continuous problem and the finite element discretization, and
in this way obtain rigorous proofs of parameter-independent GMRES convergence, see, e.g.,
\cite{Benzi:2011ku,Loghin:2004cp,Starke:1997un,Gander:2015jf,Hannukainen:2011tya,Hannukainen:2013hl}.  Other convergence  bounds for GMRES based on the field of values
are surveyed in \cite{Liesen:2018us}.

We close this section by recalling the main result in \cite{Gander:2015jf}
restricted to the kind of problems and domains that we are considering here
(i.e., the interior impedance problem on
convex
polyhedral domains discretized with P1 finite elements).
We remark that the analysis in \cite{Gander:2015jf} includes also  the exterior
scattering
problem and more general domains (star-shaped
domains).

\begin{theorem}[Theorem 1.5 in \cite{Gander:2015jf}]
Let $\Omega$ be a convex polyhedron and suppose that the matrices $\A$ and
$\A_{\eps}$ result from the discretization of the Helmholtz and shifted Laplace
problems \ref{eq:helmholtz} and \ref{eq:slaplace2} with P1 finite elements on a
quasi-uniform
sequence of meshes $\{T_h\}_{h
\geq 0}$. Let $\eps \lsim k^2$, and $k_0, C>0$. Then, there exist 
constants $C_1,C_2$ (independent of $h,k, \eps$ but depending on $k_0, C$) 
such that, for
$k > k_0$ with $hk^2 \geq C$, 
\begin{align*}
\|\I-\A\A_{\eps}^{-1}\| &\leq C_1\frac{\eps}{k}\\
\|\I-\A_{\eps}^{-1}\A\| &\leq C_2\frac{\eps}{k},
\end{align*}
and the GMRES method applied to the linear systems
\begin{align*}
\A_{\eps}^{-1}\A \u = \A_{\eps}^{-1}\f,\,\,\,
 \A\A_{\eps}^{-1} \f = \g, 
\end{align*}
converges in a number of iterations independent of $k$.
\end{theorem}

\section{A two-level preconditioner for the Helmholtz equation based on the
shifted Laplacian}

In order to introduce the two-level preconditioner for the Helmholtz problem,
we first 
review the basics of the multigrid method for finite element problems,
following the presentation in \cite{Braess:2007wm}. Let $\V_H \subset \V_h$ be
a subspace of finite element functions of dimension $N_H$, corresponding to a
coarse grid $T_H \subset T_h$. We denote by $\Phi_H$ and $\Phi_H'$ the
coordinate mappings for $\V_H$ and $\V_H'$. Since
$\V_H \subset \V_h $  and $\V_h' \subset \V_H',$
the prolongation and restriction  operators $P: \V_H \to \V_h$ and $R: \V_h'
\to \V_H'$  can be defined trivially, that is, $P v = v$ for  $v \in \V_H$ and
$R f = f|_{\V_H}$ for $v \in \V_h'$.  Moreover, the coordinate mappings
satisfy
\[\C^{N_H} \overset{ \Phi_H}{\longrightarrow}\V_H \subset \V_h \overset{
\Phi_h^{-1}}{\longrightarrow} \C^{N_h},  \text{ and } \, \, \C^{N_h} \overset{
\Phi_h'}{\longrightarrow} \V_h' \subset \V_H' \overset{
\Phi_H^{'-1}}{\longrightarrow} \C^{N_h},\]
and this gives the matrix form of the prolongation and restriction operators:
\[\P = \Phi_h^{-1} \circ \Phi_H \in \C^{N_h\times N_H} \, \text{ and } \, \R =
(\Phi_H')^{-1}  \circ \Phi_h' \in \C^{N_H\times {N_h}}.\] 

 For $v_H \in \V_H$ and $f \in
\V_h'$ we have
\begin{equation}\label{eq:adj}
\langle f, P v_H \rangle = \langle f,v_H \rangle = \langle Rf, v_H \rangle,
\end{equation}
combining this relation and part (a) of Proposition \ref{thm:dual} we conclude that
 the matrix form of the restriction operator is the Hermitian transpose of
the prolongation: $\R=\P^*$. The Galerkin coarse grid matrix is defined as
\[\A_H = \R \A \P = \P^*\A \P \in \C^{N_{H} \times N_H}.\]
Using the definition of the prolongation and restriction operators, it can be
shown that $\A_H$ corresponds to the Galerkin operator in the coarse space
$\V_H$ (Lemma 9.1 in \cite{Cocquet:2017jw}). The two-grid preconditioner
$\B_{\eps}$
that we will study has the matrix form
\[ \B_{\eps} = \A_{\eps}^{-1}(\I-\A \P \A_H^{-1}\P^*) + \P \A_H^{-1}\P^*.
\]
where $\A_\eps$ is the discrete shifted Laplacian in the space $\V_h$. Note
that it is not immediate that the preconditioner is non-singular, but this is in fact the case since it has been
shown in \cite{GarciaRamos:2020fu} that a two-level preconditioner of the form
is non-singular if and only if the matrix $\P^*\A_{\eps} \P$ is non-singular,
and this holds  because of the coercivity of $a_{\eps}(\cdot,\cdot)$. We can
now state our main result.
\begin{theorem}\label{thm:main} There exists a constant $C>0$ depending only on
the
domain
$\Omega$ such that if the coarse grid size $H$ satisfies $Hk^{2} <C$ the GMRES
method in the inner product induced by the inverse mass matrix $\M^{-1}$ \eqref{eq:invmass} applied to the preconditioned
system
\begin{equation}
	\A\B_{\eps}\g=\f,
\end{equation}
converges in a number of iterations independent of the wavenumber $k$.
\end{theorem}

\textbf{Outline of the proof of Theorem \ref{thm:main}}:
To prove Theorem \ref{thm:main} we will show that for a sufficiently small $H$
the field of values $\mathcal{F}_{\M^{-1}}(\A\B_{\eps})$ is contained in a
circle centered at $1$ with radius independent of the wavenumber $k$.  The norm
induced by
$\M^{-1}$ is a natural norm to measure the residuals of
the preconditioned system since it corresponds to the dual $L^2$ norm when
$\C^{N_h}$ is identified with the space of coordinates of $\V_h'$ (see
Proposition \ref{thm:dual}, (c)).  Recalling that $\M^{-1}$ is the matrix
representation of the $L^2$ Riesz map
(Proposition \ref{thm:dual}, (b)), we let $\mathbf{\hat{u}} = \M^{-1}\g$,
$\hat{\A}_{\eps}=\A \B_{\eps} \M$, and write the preconditioned system
\begin{equation*}
\A\B_{\eps}\g=\f,
\end{equation*}
as the equivalent system
\begin{equation}\label{eq:precgalerkin}
\hat \A_{\eps} \hat{\u} = \f.
\end{equation}

The latter linear system  encodes a functional equation \[\mathcal{\hat 
A}_{\eps}\hat{u} = f,\] where $\hat{u} \in \V_h$ and $f \in \V_h'$. This form
allows
us to formulate the preconditioned problem as a variational problem in $\V_h$.
The proof of Theorem \ref{thm:main} is divided in several parts:
\begin{enumerate}
 \item[1.]  First, we show in proposition \ref{thm:twogridgalerkin} that the
linear
system of equations \eqref{eq:precgalerkin} corresponds to the  formulation of
a variational problem in the space $\V_h$ for a certain sesquilinear form
$\hat{a}_{\eps}(\cdot, \cdot)$. 

\item[2.] Next, we prove in Lemma \ref{thm:Ahat} that the form $\hat{a}_{\eps}$
and the field of values  $\mathcal{F}_{\M^{-1}}(\A\B_{\eps})$  are related by
 \begin{equation}\label{eq:foveq}
	\mathcal{F}_{\M^{-1}}(\mathbf{A}\B_{\eps}) = 
\left\{\frac{\hat{a}_{\eps}(u,u)}{(u,u)}: u \in \V_h \setminus \{0\} \right
\}.
\end{equation}

\item[3.] 
In Lemma \ref{thm:properties} we establish some properties of the sesquilinear
form $\hat{a}_\eps$. In particular, we show that there exists an operator
$J:\V_h \to \V_h$ such that \[\frac{\hat{a}_{\eps}(u,u)}{(u,u)}= 1 + i \eps
\frac{(Ju,u)}{(u,u)},\]
for all $u \in \V_h$, thus reducing the analysis of the field of values 
\eqref{eq:foveq} to that of the field of values of the operator $J$ in the
$L^2$ inner product.

\item[4.] Lemmas \ref{thm:L2bound} and \ref{thm:coerc} 
show that the form $a(\cdot,\cdot)$ is coercive when restricted to a subspace
of $\V_h$.
Here we use a variation of a duality argument due to Schatz
\cite{Schatz:1974ea}, who used it to prove existence and quasioptimality
of the Galerkin solution of a variational problem with a sesquilinear form that
satisfies a G\aa rding inequality. This argument has appeared in several forms
in the literature of finite element analysis of Helmholtz problems (see, e.g.,
\cite{Sauter:2006fz},  \cite{Esterhazy:2011js} and the references therein), and
has been used before also in the analysis of two-grid  methods for indefinite
problems
(see \cite{Hannukainen:2013hl, Cai:1992ba}).

\item[5.] Finally, we use the coercivity result for $a(\cdot,\cdot)$ in Lemma
\ref{thm:lemmaJ2} and the coercivity of $a_{\eps}(\cdot,\cdot)$ in Lemma
\ref{thm:lemmaJ3}
to give estimates that can be combined to bound the norm of $\eps J$
independently of the wavenumber, under the restriction that $H k^2 \leq C$ for
a constant $C$ independent of $k$. 
\end{enumerate}

To proceed with the first step we formulate the preconditioned problem as a
variational problem in $\V_h$
using the strategy from \cite{Hannukainen:2013hl}, outlined in the next two
propositions.

\begin{proposition}\label{thm:twogridgalerkin}
Let $a$ be the sesquilinear form from the Helmholtz problem defined in
\eqref{eq:helmform}, let $\tau: \V_h' \to ~ \V_h$ the $L^2$-Riesz map, and $P:
\V_H
\to \V_h$,  ~$R:\V_h' \to \V_H'$ be the prolongation and restriction operators
respectively. The following statements hold:
\begin{enumerate}
\item Let $Q: \V_h \to \V_H$ be the solution operator to the problem:  For
$u
\in \V_h$ find $Qu \in \V_H$ such that
\begin{equation}\label{eq:Q}  a(Qu,v)  =\langle R\tau^{-1} u,v \rangle  \text{
for all $v \in \V_H$},\end{equation}
then, the matrix form of the operator $Q: \V_h \to \V_h$ is
\begin{equation}
\Q = \P\A_H^{-1}\P^*\M \in \C^{N_h \times N_h} \end{equation}

\item  Let $N: \V_h \to \V_H$ be the solution operator to the adjoint-type
problem: For $u \in \V_h$ find  $Nu \in \V_H$ such that
\begin{equation}\label{eq:N} 
a(v,Nu) = a(v,u) \text{ for all $v \in \V_H$,}
\end{equation}
then, the matrix form of the operator $N:\V_h \to \V_h$ is
\[ \N = \P \A_H^{-*} \P^* \A^*. \]

\item Let $J: \V_h  \to \V_h$ be the solution operator to the problem: For
$u\in \V_h$ find $Ju \in \V_h$ such that for all $v \in \V_h$
\begin{equation}\label{eq:J}
a_{\eps}(Ju,v) = \langle \tau^{-1} u, (I-N)v \rangle,	\end{equation} 
then, the matrix form of the operator $J$ is
\[\J = \A_{\eps}^{-1}(\I-\N^*)\M= \A_{\eps}^{-1} (\I- \A \P \A_H^{-1}\P^* ) \M.
\]
\end{enumerate}
\end{proposition}

\begin{proof} To prove (a), let $u \in \V_h$, $\u = \Phi_h^{-1}u
\in \C^{N_h}$ and $\w = \Phi_H^{-1}Qu \in \C^{N_H}$, recalling that $\M$ is the
matrix representation of the inverse Riesz map $\tau^{-1}$ we have that the
condition \eqref{eq:Q} is equivalent to
\[ \vc^* \A _H \w  = \vc^*\P^*\M\u  \text{ for all }\vc \in \C^{N_H},\]
this gives $\w = \A_H^{-1}\P\M\u$ and $\Q\u = \P \w = \P \A_H^{-1}\P\M\u$.

To show (b), let $u \in \V_h$, $\u = \Phi_h^{-1}u \in \C^{N_h}$ and $\tilde{\w}
= \Phi_H^{-1}Nu \in \C^{N_H}$. Then condition \eqref{eq:N} is equivalent to
\[ \tilde \w^*\A _H \vc = \u^*\A \P \vc  \text{ for all }\vc \in \C^{N_H},\]
Therefore $\A_H^* \tilde \w = \P^* \A^* \tilde \u$, so $\tilde \w = \A_H^{-*}
\P^* \A^{*} \u$ and $\N \u = \P \tilde \w = \P \A_H^{-*} \P^* \A^* \u$.

To prove (c), let $u \in \V_h$, $\u = \Phi_h^{-1}u \in \C^{N_h}$ and $\J\u =
\Phi_h^{-1}Ju \in \C^{N_h}$. Using the results of parts (a) and (b), we see
that \eqref{eq:J} is equivalent to
\[ \vc^*\A_{\eps} \J \u = ((\I-\N^*)\vc)^*\M\u = \vc^*(\I-\N^*)\M\u \text{ for
all }\vc \in \C^{N_h},\]
hence $\J\u = \A_{\eps}^{-1}(\I-\N^*)\M \u= \A_{\eps}^{-1}(\I-\A \P \A_{H}^{-1}
\P^*)\M \u$.
\end{proof}

\begin{lemma} \label{thm:Ahat}
Let $\A, \B_{\eps}, \M$ be the discrete Helmholtz operator, the two-grid
preconditioner and the mass matrix in $\V_h$ respectively, and define $\hat
\A_{\eps} = \A \B_{\eps} \M$.  The following properties hold:
\begin{enumerate}
\item If $\hat{a}_\eps: \V_h \times \V_h \to \C$ is defined as
\begin{equation}
\hat{a}_{\eps}(u,v)= a((J+Q)u,v),
\end{equation}
then, for $\u, \vc \in \C^{N_h}$ and $u = \Phi_h \u, v= \Phi_h \vc \in \V_h$:
\[ \u^*\hat{\A}_{\eps}\vc = \hat{a}_{\eps}(u,v).\]

\item Given $f \in  \V_h'$, consider the problem: 
\begin{equation}\label{eq:precgal} \text{Find $u \in \V_h$ such that } \hat
a(u,v) = \langle f,v \rangle \text{ for all } v \in \V_h.\end{equation}
Then, the preconditioned system
$\hat{\A}_{\eps} \u = \f$ is the linear algebraic formulation of
\eqref{eq:precgal}.
\end{enumerate}	
\end{lemma}

\begin{proof} Part (a) follows from the definition of $\hat \A_{\eps}$ and the
matrix representations of $J$ and $Q$ given in the previous proposition. Part
(b) is a straightforward consequence of (a).
\end{proof}

\begin{lemma}\label{thm:ahat}
Let $\A \in \C^{N_h}$ be the discrete Helmholtz operator of the Galerkin
problem in $\V_h$, $\M \in \C^{N_h} $ the mass matrix for the finite element
space $\V_h$  and $\A_{\eps} \in \C^{N_h}$ the discrete shifted Laplacian. Let
$\B_{\eps}$ be the two-grid preconditioner defined by
\[\B_{\eps} = \A_{\eps}^{-1}(\I-\A \P \A_H^{-1}\P^*) + \P \A_H^{-1}\P^*,\]
and $\hat a_\eps$ the sesquilinear form defined in Proposition \ref{thm:ahat}
(a). Then, the  field of values of $\A \B_\eps$ in the inner product induced by
$\M^{-1}$ is 
\begin{align}\label{eq:fova}
\mathcal{F}_{\M^{-1}}(\mathbf{AB}_\eps) &= 
\left\{\frac{\hat{a}_{\eps}(u,u)}{(u,u)}: 0 \neq u \in \V_h \right \}.
\end{align}
\end{lemma}
\begin{proof} Let $z = \frac{(\A \B_{\eps} \g,
\g)_{\M^{-1}}}{\|\g\|_{\M^{-1}}^2} \in
\mathcal{F}_{\M^{-1}}(\mathbf{A}\B_{\eps})$, $\u = \M^{-1}\g$ and $u= \Phi_h
\u.$ Then,
\begin{align*}
z = \frac{(\A \B_{\eps}\g, \g)_{\M^{-1}}}{\|\g\|_{\M^{-1}}^2}&= \frac{\g^*
\M^{-1}\A \B_{\eps} \g}{\|\g\|_{\M^{-1}}^{2}}\\ &=
\frac{(\M^{-1}\g)^*\A\B_{\eps}\g}{\|\g\|_{\M^{-1}}^2}\\
&= \frac{\u^*\A\B_{\eps}\M \u}{\|\M \u\|_{\M^{-1}}^{2}}=
\frac{\u^*\A\B_{\eps}\M
\u}{(\M\u)^*\M^{-1}\M\u}=\frac{\u^*\hat{\A}_{\eps}\u}{\u^*\M\u}=
\frac{\hat{a}_{\eps}(u,u)}{(u,u)},
\end{align*}
where in the last step we used part (a) of Lemma \ref{thm:Ahat}. Since the
correspondence \[\g \mapsto \u = \M^{-1} \g \mapsto u= \Phi\u\] is a bijection
between $\C^{N_h} \setminus \{\mathbf{0}\}$ and $\V_h \setminus \{0\}$, the 
equality of sets \eqref{eq:fova} follows.
\end{proof}

\begin{proposition}[Properties of $\hat a_{\eps}, Q, N$] \label{thm:properties}
Let $\hat{a}_{\eps}: \V_h \times \V_h \to \C$ be the sesquilinear form defined
in Lemma \ref{thm:ahat}, $\tau: \V_h' \to \V_h$ the $L^2$ Riesz map and $Q,N,J$
the operators introduced in Proposition \ref{thm:twogridgalerkin}.
\begin{enumerate}
	\item For all $u \in \V_h$ \[a(Qu,u) = \langle\tau^{-1} u, Nu \rangle. \]
	\item For all $u \in \V_h$, $v \in \V_H$
	   \[a(v,(I-N)u)=0.\]
	 \item For all $u \in \V_h$
	 \begin{equation}
	 \hat{a}_{\eps}(u,u)=(u,u) + i\eps (Ju,u).	
	 \end{equation}
\end{enumerate}
\end{proposition}

\begin{proof} We begin with part (a). From the definition of $Q$ and $N$ we
have, for all $u \in \V_h$, and $w,v \in \V_H$:
\begin{align} 
a(v,u)&=a(v,Nu), \label{eq:N1}\\
a(Qu,w) &= \langle R \tau^{-1}u,w \rangle \label{eq:Q1}.
\end{align}
Substituting $v=Qu$ in \eqref{eq:N1} we obtain for  $u \in \V_h$
\[a(Qu,u)=a(Qu,Nu),\]
and setting $w=Nu$ in \eqref{eq:Q1} gives for $u \in \V_h$
\[a(Qu,Nu) = \langle R \tau^{-1}u,Nu \rangle = \langle \tau^{-1}u,Nu \rangle.\]
therefore $a(Qu,u)=\langle \tau^{-1}u,Nu \rangle$ holds for all $u \in \V_h$.
The statement (b) is a consequence of the definition of the operator $N$. To
prove (c), recall that $a(u,v)=a_\eps(u,v)+i\eps(u,v)$, then using the
definition of $Q , J$ and part (a) we get
\begin{align*}
\hat{a}_{\eps}(u,u) = a((J+Q),u) &= a(Ju,u) + a(Qu,u)\\
&=a_{\eps}(Ju,u)+i\eps(Ju,u) + a(Qu,u)\\
&=  \langle \tau^{-1}u,(I-N)u \rangle + i\eps(Ju,u) +  \langle \tau^{-1} u,Nu
\rangle\\
&= \langle \tau^{-1}u, u \rangle + i\eps(Ju,u)\\
&= ( u, u) + i\eps(Ju,u).
\end{align*}
\end{proof}

The following two lemmas show that the sesquilinear form $a$ is coercive when
restricted to the range of the operator $N$.

\begin{lemma}[Bound on $L^2$ norm of $I-N$]\label{thm:L2bound} For every $u \in
\V_h$ 
\begin{equation}\label{eq:L2bound}
	\| (I-N) u \|_{L^2(\Omega)} \lsim Hk \|(I-N)u\|_{1,k,\Omega}.\end{equation}
\end{lemma}

\begin{proof}
We use a duality argument. Let $v=(I-N)u$  and $\phi$ the solution to the
problem: Find $\phi \in H^{1}(\Omega)$ such that
\[a(\phi,w) = (  v,w ) \text{ for all $w \in H^{1}(\Omega)$}.\]
If $\Pi_{SZ}: H^{1}(\Omega) \to \V_H$ is the Scott-Zhang interpolation
operator, for $w=v$ we have
\begin{align*}
\|(I-N)u\|^2_{L^{2}(\Omega)}=|(v,w)|  &= |a(\phi, (I-N)u)|\\
&=
  |a(\phi - \Pi_{SZ} \phi,(I-N)u)| \,\,\, \text{ (Lemma \ref{thm:properties},
part (b))}\\
&\lsim \| \phi - \Pi_{SZ} \phi \|_{1,k,\Omega} \|(I-N)u\|_{1,k,\Omega}
\, \text{(continuity of $a$)}.
\end{align*}
Since $(I-N)u \in H^1(\Omega)$ and $\Omega$ is a convex polygon we have
$\phi\in H^2(\Omega)$ (Theorem 2, section 6.3.1 in \cite{Evans:2016wj}), and
using the interpolation estimate \eqref{eq:szint1} and the stability estimates
for the Helmholtz problem (Theorem \ref{thm:stab}) gives
\begin{align*}
	\| \phi - \Pi_{SZ} \phi \|_{1,k,\Omega} &\leq H \|\phi\|_{H^2(\Omega)} +
Hk\|\phi\|_{H^{1}(\Omega)}\\
	& = H(|\phi|_{H^2(\Omega)} + \|\phi\|_{H^1(\Omega)}) + 
Hk\|\phi\|_{H^{1}(\Omega)}\\
	& \lsim H(k\|(I-N)u\|_{L^2(\Omega)} +\|(I-N)u\|_{L^2(\Omega)}) +
Hk\|(I-N)u\|_{L^2(\Omega)}\\
	&\lsim Hk \|(I-N)u\|_{L^2(\Omega)}.
\end{align*}
Combining these estimates we obtain
\begin{align*}
\| (I-N)u \|_{L^2(\Omega)}^2 \lsim Hk \|(I-N)u\|_{L^2(\Omega)}
\|(I-N)u\|_{1,k,\Omega},
\end{align*}
and dividing both sides of the inequality by $\|(I-N)u\|_{L^2(\Omega)}$ yields
\eqref{eq:L2bound}.
\end{proof}

In the next lemma we show that if the coarse grid is sufficiently fine, the
bilinear form $a$ is coercive when restricted to the range of the operator
$I-N$.  

\begin{lemma}\label{thm:coerc} There exist constants $C, \alpha$
independent of $h,H,k$ such that \linebreak if $Hk^2<C$ 
\[ \alpha \|(I-N)u\|^2_{1,k,\Omega} \leq \Re{a((I-N)u,(I-N)u)}.\]
\end{lemma}
\begin{proof}
The sesquilinear form $a$ satisfies the G\aa rding inequality
\[\Re{a}((I-N)u,(I-N)u) =\|(I-N)u\|^{2}_{1,k,\Omega} -2k^2
\|(I-N)u\|_{L^2(\Omega)}^2. \]
Combining this with Lemma \ref{thm:L2bound} gives
\begin{align*}
\|(I-N)u\|^{2}_{1,k,\Omega} -2k^2 \|(I-N)u\|_{L^2(\Omega)} &\geq
\|(I-N)u\|_{1,k,\Omega}^2 - |\tilde{C}H^2k^4\|(I-N)u\|_{1,k,\Omega}^2\\
&= (1-\tilde{C}H^2k^4) \|(I-N)u\|_{1,k,\Omega}^2
\end{align*}
where the constant $\tilde{C}$ comes from the estimate in  Lemma
\ref{thm:L2bound}. Let $\alpha \in (0,1)$ and define
$C = [(1-\alpha) \tilde{C}^{-1}]^{1/2} > 0$. It is easy to see that if
$Hk^2<C$
we
have
$(1-\tilde{C}H^2K^4) > \alpha$, therefore we have
\begin{equation*}
\alpha \|(I-N)u\|^2_{1,k,\Omega}  \leq \Re{a((I-N)u,(I-N)u)}.
\end{equation*}
\end{proof}

\begin{lemma}[Using coercivity to bound
$\|(I-N)Ju\|_{1,k,\Omega}$]\label{thm:lemmaJ2} Suppose that $H$ satisfies the
requirements of the previous lemma. Then, for all $u \in \V_h$
\begin{equation}\label{eq:I-N}
\|(I-N)u\|_{1,k, \Omega} \lsim (1 + Hk) \|u\|_{1,k, \Omega}
\end{equation}
\end{lemma}
\begin{proof}
Let $H$ be such that $Hk^2$ is sufficiently small, i.e. $Hk^2 < C$ where $C$ is
the constant from the previous lemma. For $u \in \V _h$, let $\Pi_{SZ} u$ be
the
Scott-Zhang interpolant of $u$ in $\V _H$. Combining  Lemma \ref{thm:coerc},
the orthogonality
condition
in part (b) of Lemma \ref{thm:properties} and the
continuity of $a$ we have
\begin{align*}
\|(I-N)u\|^2_{1,k}  &\lsim \Re{a((I-N)u,(I-N)u)} \,\, \text{ (Lemma
\ref{thm:coerc})} \\
&=\Re a(u - \Pi_C u,(I-N)u)  \,\, \text{(Lemma \ref{thm:properties}, (b))}
\\
&\lsim \|u-\Pi_C u \|_{1,k} \|(I-N)u\|_{1,k} \,\, \text{(by continuity of
$a$).}
\end{align*}
With the estimate \eqref{eq:szint2} for the Scott-Zhang interpolation
operator, we
obtain
\begin{align*}
	\| u - \Pi_{SZ} u \|_{1,k,\Omega} & \lsim (1 + Hk) \|u\|_{H^{1}(\Omega)} \\
& \leq
(1 + Hk) \|u\|_{1,k,\Omega}.\end{align*}
Therefore,
\begin{align*}
\|(I-N)u\|^2_{1,k,\Omega}  &\lsim \|u-\Pi_C u \|_{1,k,\Omega}
\|(I-N)u\|_{1,k,\Omega}
\\
&\lsim (1 + Hk) \|u\|_{1,k,\Omega} \|(I-N)u\|_{1,k,\Omega}.
\end{align*}
Dividing both sides by $\|(I-N)u\|_{1,k,\Omega}$ gives \eqref{eq:I-N}.
\end{proof}

\begin{lemma}[Bound on $L^2$ norm of $J$]\label{thm:lemmaJ3} For all $u \in
\V_h$ we have
\begin{equation}
\| J u \|_{L^2(\Omega)} \lsim \frac{Hk^2}{\eps} \|u\|_{L^2(\Omega)}. 
\end{equation}
\end{lemma}

\begin{proof}
	We estimate as follows:
	\begin{align*}
		\|Ju\|_{1,k,\Omega}^2 &\leq \alpha^{-1}\frac{k^2}{\eps}|a_{\eps}(Ju,Ju)| \,\,
\text{ (by coercivity of $a_\eps$)}\\
		&= \alpha^{-1}\frac{k^2}{\eps} |\langle \tau^{-1}u, (I-N)Ju\rangle | \,\,
\text{ (by definition of $J$)}\\
		&=\alpha^{-1}\frac{k^2}{\eps}| (u,(I-N)Ju)_{L^2(\Omega)}| \,\, \text{ (by
definition of $\tau$)}\\
		&\leq \alpha^{-1}\frac{k^2}{\eps} \|u\|_{L^2(\Omega)}
\|(I-N)Ju\|_{L^2(\Omega)} \,\,\, \text{ (by the Cauchy-Schwarz ineq.)}\\
		&\lsim \frac{Hk^3}{\eps}\|u\|_{L^2(\Omega)}\|(I-N)Ju\|_{1,k,\Omega} \,\,\,
\text{(by Lemma \ref{thm:L2bound})} \\
		& \lsim \frac{Hk^3}{\eps} (1+Hk) \|u\|_{L^2(\Omega)} \|Ju\|_{1,k} \,\,\,
\text{(by Lemma \ref{thm:lemmaJ2})},
	\end{align*}
dividing each side of the inequality by $ \|Ju\|_{1,k}$ we obtain
\[\|Ju\|_{1,k,\Omega}  \lsim \frac{Hk^3}{\eps} (1+Hk) \|u\|_{L^2(\Omega)} \]
and, since $Hk  \lsim 1$,
\[\|Ju\|_{L^2(\Omega)} \leq k^{-1} \|Ju\|_{1,k,\Omega} \lsim
\frac{Hk^2}{\eps}\|u\|_{L^2(\Omega)}.\] 
\end{proof}	
We can now prove our main result.
\begin{proof}[Proof of Theorem \ref{thm:main}] Combining Lemma \ref{thm:ahat}
and part (c) of Proposition \ref{thm:properties}, we have that the field of
values of the preconditioned matrix $\A\B_\eps$ in the inner product induced by
$\M^{-1}$ is
the
set
\begin{equation}
\mathcal{F}_{\M^{-1}}(\A\B_{\eps}) =  \left\{ 1 + i\eps\frac{(Ju,u)}{(u,u)}: u
\in \V_h\setminus \{0\} \right\}.
\end{equation}
By Lemma \ref{thm:lemmaJ3}, if $\eps \lsim k^2$ and $Hk^2$ is sufficiently
small we have
\[\eps \|Ju\| \lsim Hk^2 \|u\|,\]
therefore, under this restriction on $H$, we have
\[\eps\frac{|(Ju,u)|}{(u,u)} \leq \eps\frac{\|Ju\|\|u\|}{\|u\|^2} \lsim Hk^2,\]
so choosing $Hk^2$ sufficiently small the field of values
$\mathcal{F}_{\M^{-1}}(\A\B_{\eps})$ lies inside a circle centered at $1$ that
does not contain the origin, with radius independent of the
wavenumber $k$. 
Therefore,  under this restriction on the coarse grid size, the distance of the
field of values to the origin
$\nu_{\M^{-1}}(\A\B_{\eps})$ is
independent
of
 $k$. Moreover, the inequality (see Chapter 1 of
\cite{Horn:1994tx})
\[  \|\A\B_{\eps}\|_{\M^{-1}} \leq 2 \max_{z \in \mathcal{F}_{\M^{-1}} (\A
\B_{\eps})} |z|, \]
implies that the norm $\|\A \B_{\eps}\|_{\M^{-1}}$ is bounded independently of
$k$ as well. Therefore, the quantity
\[\frac{\nu_{\M^{-1}}(\A\B_{\eps})}{\|\A\B_{\eps}\|_{\M^{-1}}}\]
is bounded away from zero independently of $k$. Using the residual
bound \eqref{eq:elmanbound} we conclude that, for  $Hk^2$ sufficiently small,
if the GMRES method in the inner product induced by $\M^{-1}$ is applied to the
linear system
\[\A\B_{\eps} \g = \f,\]
the number of iterations required to obtain a reduction of the relative
residual by a fixed tolerance is bounded by a constant, independent of the 
wavenumber $k$.
\end{proof}

In the next corollary we show that Theorem \ref{thm:main}
also
holds for the Euclidean inner product in the case of quasi-uniform
meshes.

\begin{corollary}\label{thm:corollarynorm}
 Suppose, in addition to the hypothesis of Theorem \ref{thm:main},
that the
sequence of meshes $\{T_h\}_{h >0}$ is quasi-uniform. Then,
there exists a
constant $C>0$
depending
only
on
the domain $\Omega$ such that if the coarse grid size $H$ satisfies $Hk^{2} \leq  C$
the GMRES method in the Euclidean inner product applied to the preconditioned
system
\begin{equation*}
	\A\B_{\eps}\g=\f,
\end{equation*}
converges in a number of iterations independent of the wavenumber
$k$.
\end{corollary}

\begin{proof} Recall that for a sequence of quasi-uniform meshes  the
following norm
equivalence  holds:
\[ \| \vc\|_{\M} \sim h^{d/2} \|\vc\|_{\I}, \]
for all  $\vc \in
\mathbb{C}^{N_h}$ and  $ h > 0$,
with the hidden constants  independent of $h$ \cite{Braess:2007wm}.

Using this fact
and
the
characterization of the norm $\|\cdot\|_{\M^{-1}}$ as the dual norm of
$\|\cdot\|_{\M}$ (Theorem \ref{thm:dual}, part (c)), it can be shown that \[ \|
\f\|_{\M^{-1}}
\sim
h^{-d/2} \|\f\|_{\I}, \]
for all  $\f \in
\mathbb{C}^{N_h}$ and  $ h > 0$, with the hidden constants  independent
of $h$.  A straightforward computation shows that 
\[\A \B_\eps =  \I + i \eps \hat\J,\]
 where $\hat{\J} = \M \A_{\eps}^{-1}( \I- \A\P
\A_{H}^{-1}\P^* )$, therefore, the field of
values
of
$\mathbf{\A\B_{\eps}}$
in
the Euclidean inner product equals
\begin{equation}\label{eq:fovI}
\mathcal{F}_{\I}(\A\B_{\eps}) = \left\{ 1 + i \eps \frac{( \hat{\J} \f,
\f)_{\I}}{(\f,\f)_{\I}}:
\mathbf{0} \neq \f \in \mathbb{C}^{N_h} \right\}.
\end{equation}
 Using the norm equivalence
between
$\|\cdot \|_{\I}$ and $\|\cdot \|_{\M^{-1}}$ from above we have
\begin{equation} \label{eq:normsim} \| \hat{\J} \|_{\I}  \sim  \| \hat{\J}
\|_{\M^{-1}},
\end{equation}
and combining part (c) of Proposition
\ref{thm:twogridgalerkin} with part (b) of Proposition
\ref{thm:dual}, we have that the matrix $\hat{\J}$ is the representation
of
the
operator $ \hat{J} = \tau^{-1} J \tau $ 
and it follows from  Theorem 5.9 that
\begin{equation}\label{eq:norm}
 \| \hat{\J} \|_{\M^{-1}} = \sup_{\f \in \mathbb{C}^{N_h}}
\frac{
\|\hat{\J}
\f\|_{\M^{-1}}
 }{\| \f\|_{\M^{-1}}}  =  \sup_{\u \in \mathbb{C}^{N_h}} \frac{ \|\J
\u\|_{\M}
 }{\| \u\|_{\M}} = \sup_{u \in \V_h} \frac{ \|Ju\|_{L^2(\Omega)}
 }{\| u\|_{L^2(\Omega)}} \lsim \frac{Hk^2}{\eps}. \end{equation}
Therefore, for all $\mathbf{0} \neq \f \in \C^{N_h}$ we have 
\[ \eps \frac{|( \hat{\J} \f,
\f)_{\I}|}{|(\f,\f)_{\I}|} \leq \eps  \frac{\|\hat{\J} \f\|_{\I} \|\f\|_{\I}
}{\|\f\|_{\I}^2} \leq \eps \frac{\|\hat{\J}\|_{\I}
\|\f\|_{\I}^2}{\|\f\|_{\I}^2}
\lsim Hk^2, \]
where we have used the Cauchy-Schwarz inequality in the first step and
\eqref{eq:normsim} together with \eqref{eq:norm} in the last step. Using
\eqref{eq:fovI}, we
conclude that choosing $Hk^2$ sufficiently small the field of values
$\mathcal{F}_{\I}(\A\B_{\eps})$ lies inside a circle centered at $1$ that
does not contain the origin, with radius independent of  the
wavenumber $k$. The rest of the proof is similar to the last part of the proof
of Theorem 5.9.
\end{proof}

\section{Numerical Experiments}

In this section we present the results of some numerical experiments that illustrate our theoretical results. The experiments were performed using MATLAB 2017b on a Macbook Pro with a 2,4 GHz Intel Core i5 processor. For the discretization with finite elements we have used the package $i$FEM \cite{Chen:2009wx}.

\subsection*{Experiment 1}
In our first experiment we study the Helmholtz
problem~\eqref{eq:helmholtz} on
the domain $\Omega =(0,1)$. Although this problem leads to linear systems that
are small and do not need 
to be solved
with
iterative methods, we use them here for illustrative purposes since for higher dimensional problems and large
values of $k$ the
computation of the field of values is very expensive. 
According to our theory, we choose for the
discretization
the
number
of interior
gridpoints for the coarse mesh equal to $\lceil
\frac{k^2}{2}
\rceil$
which
leads to a coarse problem
of
size $N_H = \lceil \frac{k^2}{2} \rceil +2$ and a fine problem of size $N_h =
2N_H-1$.  We plot 
the field of values of the
 matrices $\A\A_{\eps}^{-1}$ and $\A\B_{\eps}$ using the method of Johnson
\cite{Johnson:1978ha},
for increasing wavenumbers $k$ and various choices
of $\eps$. The main point of this experiment is to show that for
increasing
wavenumbers $k$ and a shift
$\eps \sim k^2$ the set
$\mathcal{F}_{\I}(\A\B_{\eps})$ remains
bounded
away
from the origin, as predicted by the theory, in contrast to
$\mathcal{F}_{\I}(\A\A_{\eps}^{-1})$,
which
moves closer to the origin as $k$ is increased. 
The results of this experiment are shown in Figures \ref{fig:fov1}
and \ref{fig:fov2}. Note that in this case the field of values $\mathcal{F}_{\I}(\A\B_{\eps})$ is practically equal to a single point.

We repeat this computation choosing a number
of interior
gridpoints for the coarse mesh equal to $\lceil
\frac{k^{3/2}}{2}
\rceil$,
which
leads to a coarse problem of
size $N_H = \lceil \frac{k^{3/2}}{2} \rceil +2$ and a fine problem of size $N_h=
2N_H-1$. The results of this experiment are shown in Figures
\ref{fig:fov3} and \ref{fig:fov4}. We see that that under a less restrictive condition on the
meshsize the field of values of
$\A\B_{\eps}$
remains
bounded away from zero as $k$ is increased.  This is not predicted by our
theory,
but can be explained from the fact that it has been shown in
\cite{Ihlenburg:1995eta} that the
condition $N_h \sim k^{3/2}$ is sufficient to the  obtain a
'pollution-free'
solution to the Helmholtz problem with the Galerkin method in 1-D. However,
this has not been proved in higher dimensions.

\begin{figure}[t]
\centering
\begin{subfigure}[b]{0.48\textwidth}
\centering
\includegraphics[width=0.95\textwidth]{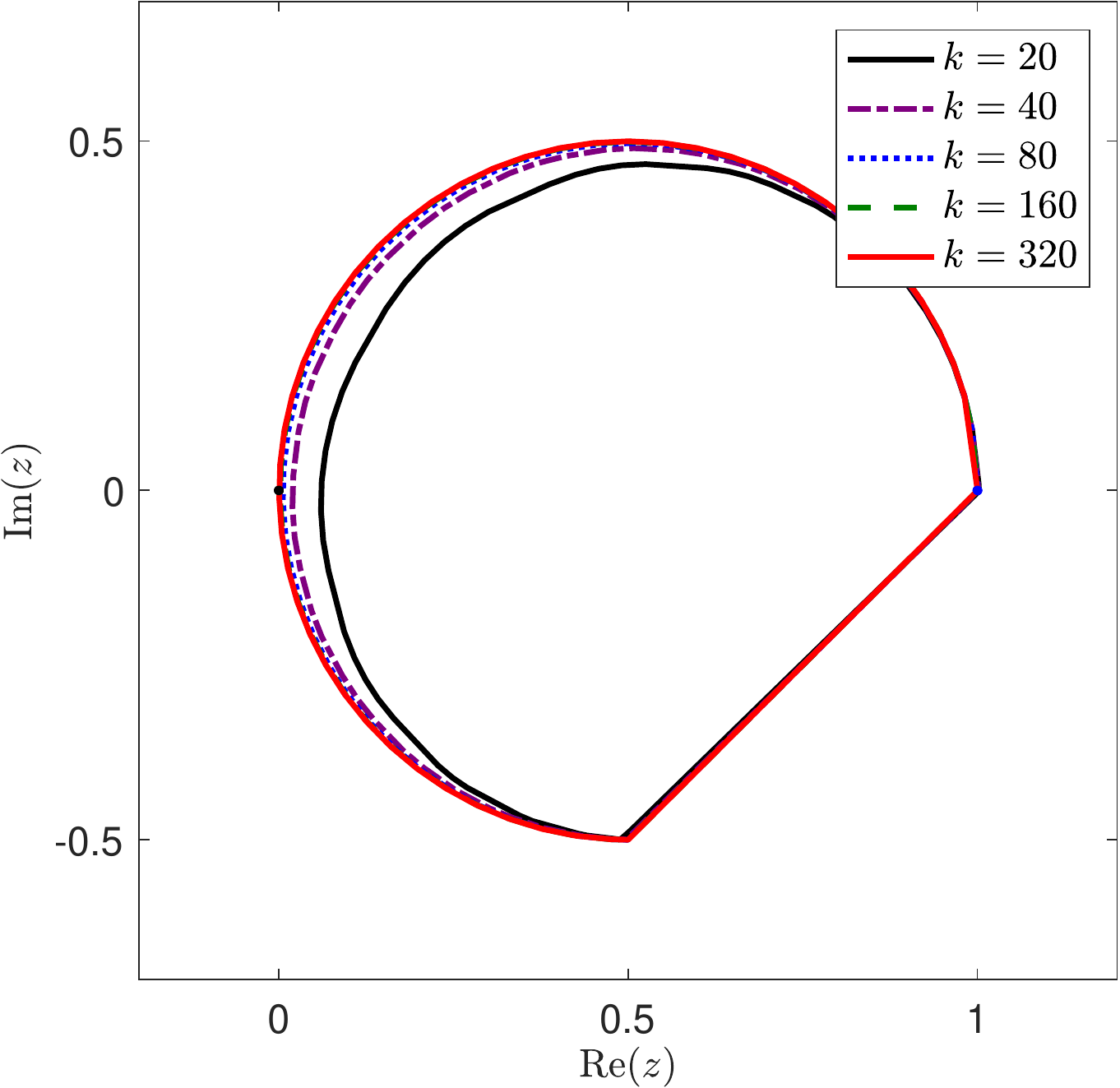}\end{subfigure}
\begin{subfigure}[b]{0.48\textwidth}
\centering
\includegraphics[width=0.95\textwidth]{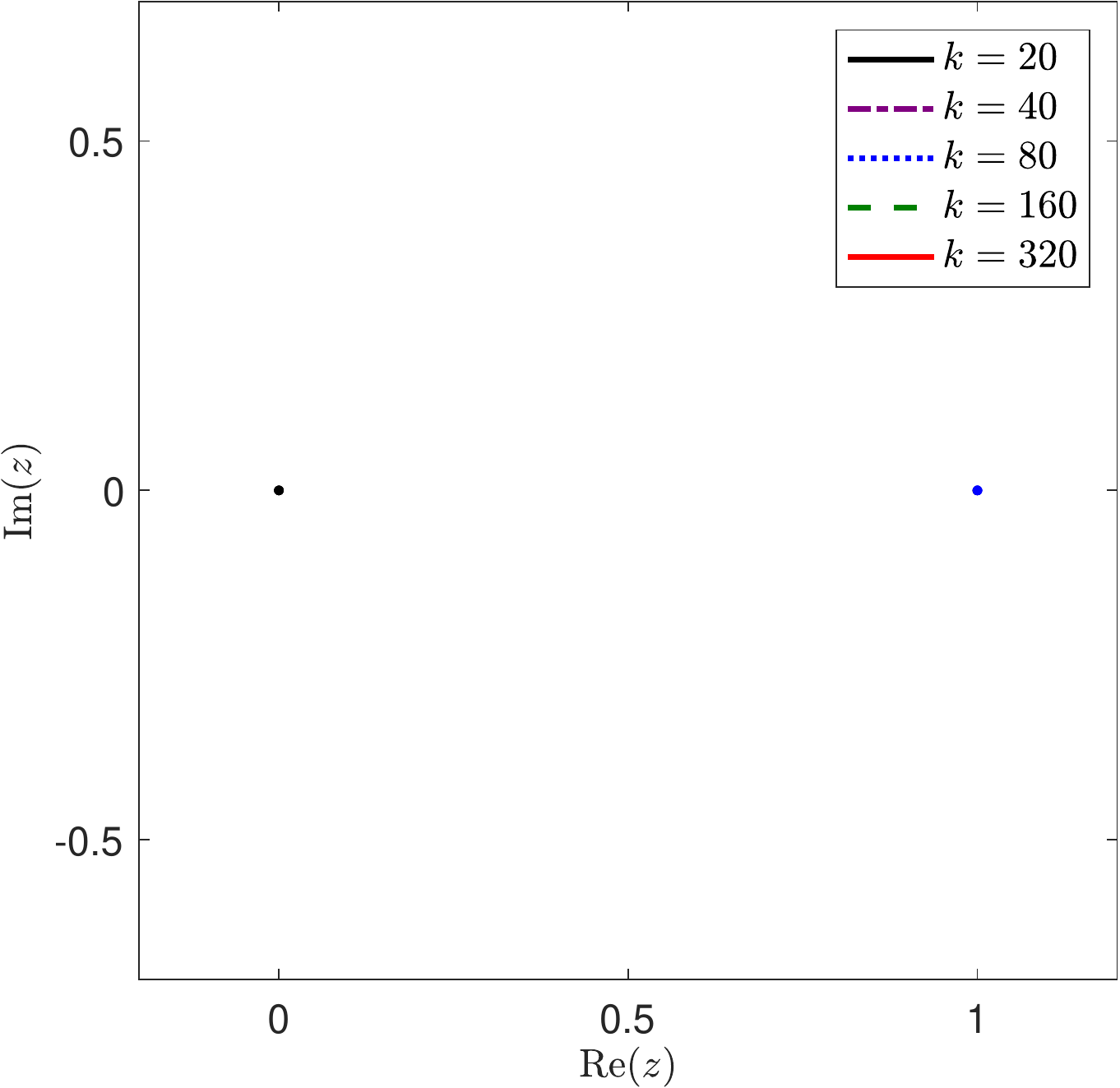}\end{subfigure}
\caption{Field of values of $\A\A_{\eps}^{-1}$ (left) and $\A\B_{\eps}$
(right)
for a 1D Helmholtz problem and various values of $k$, with $\eps = k^2$
and
$N_H \sim k^2$.}
\label{fig:fov1}
\end{figure}

\begin{figure}[t]
\centering
\begin{subfigure}[b]{0.49\textwidth}
\centering
\includegraphics[width=0.95\textwidth]{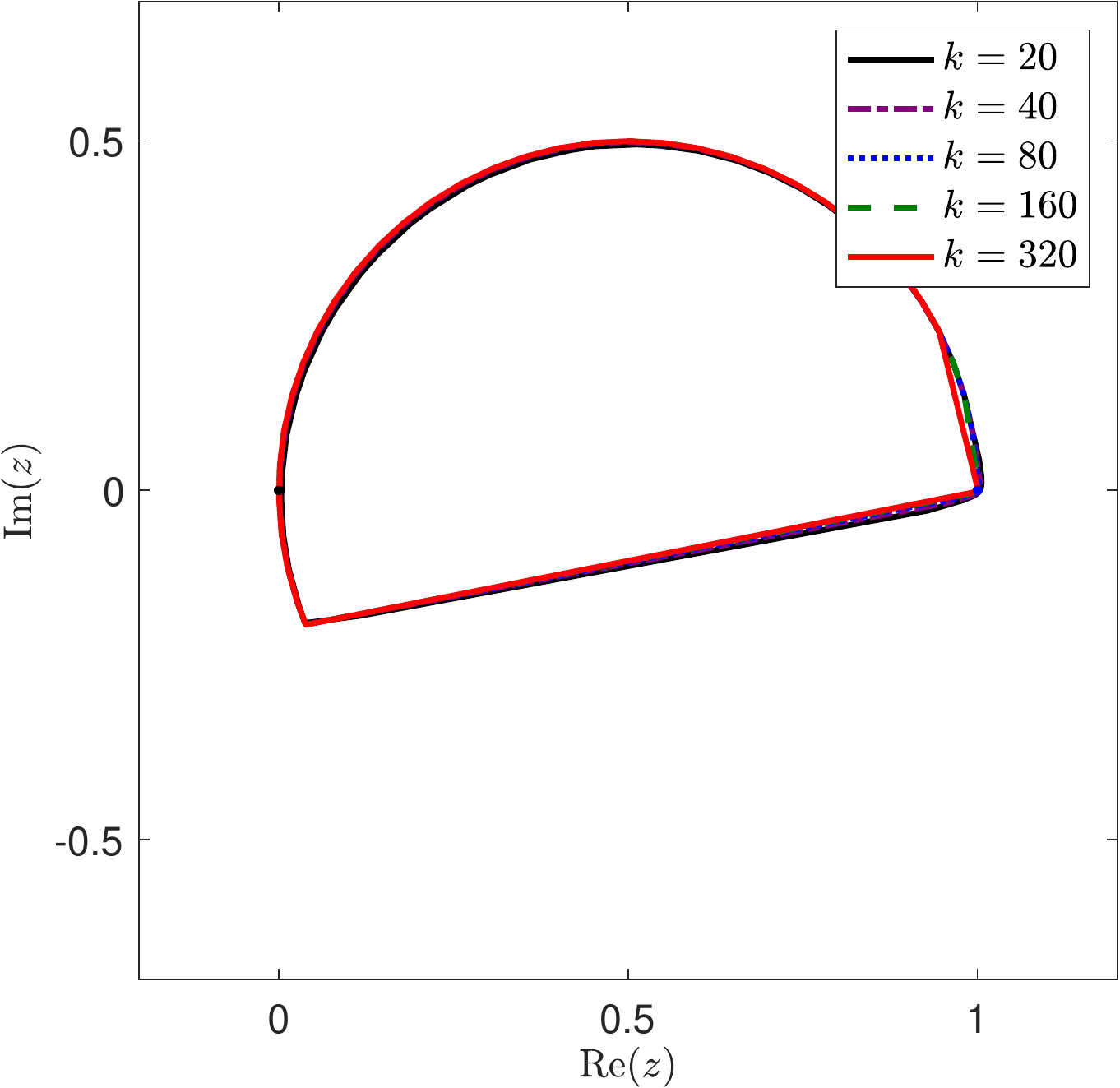}
\end{subfigure}
\begin{subfigure}[b]{0.49\textwidth}
\centering
\includegraphics[width=0.95\textwidth]{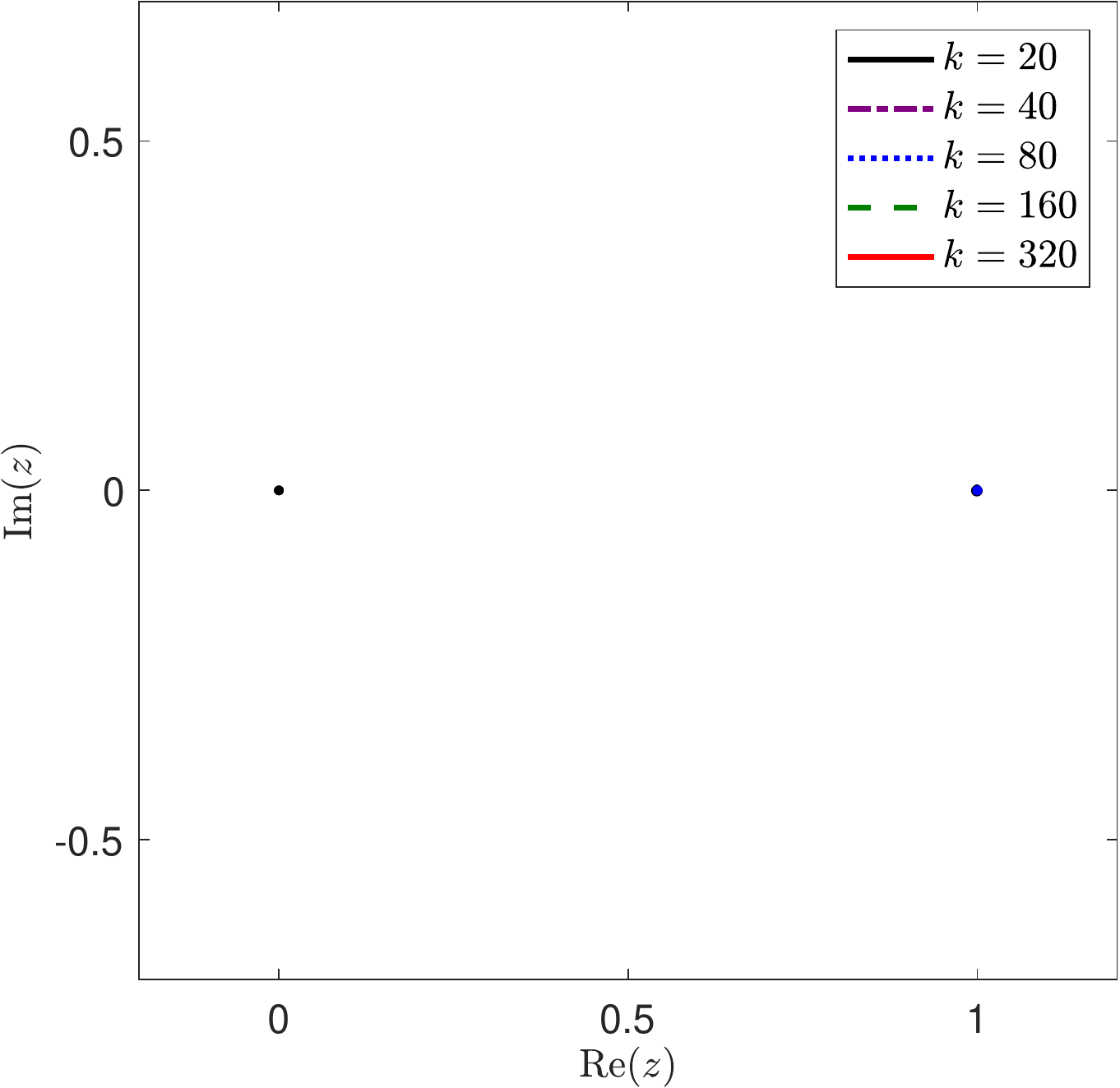}
\end{subfigure}
\caption{Field of values of $\A\A_{\eps}^{-1}$ (left) and $\A\B_{\eps}$
(right)
for a 1D Helmholtz problem and various values of $k$, with $\eps = 5k^2$
and
$N_H \sim k^2$.}
\label{fig:fov2}
\end{figure}

\begin{figure}[t]
\centering
\begin{subfigure}[b]{0.49\textwidth}
\centering
\includegraphics[width=0.95\textwidth]{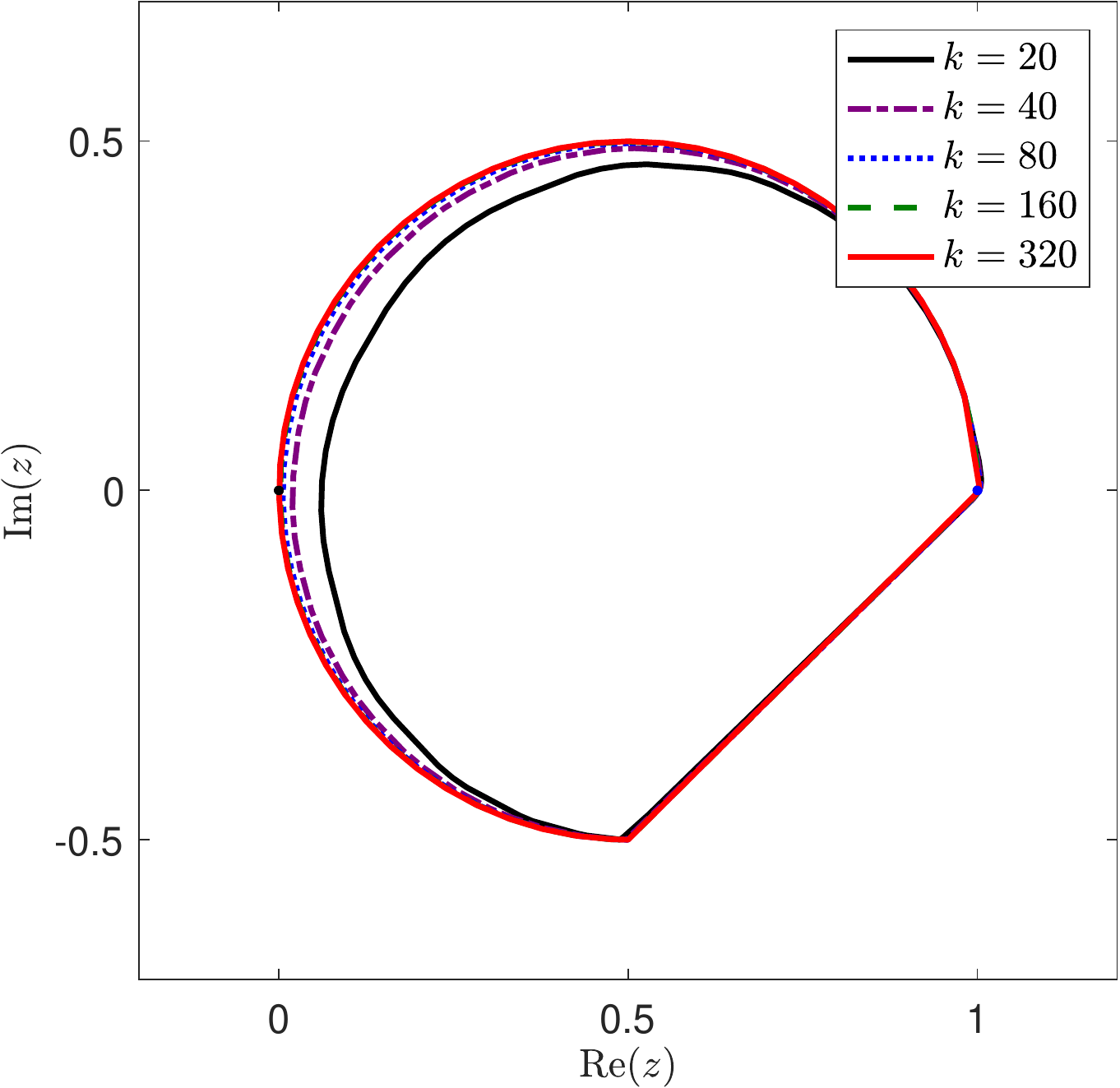}
\end{subfigure}
\begin{subfigure}[b]{0.49\textwidth}
\centering
\includegraphics[width=0.95\textwidth]{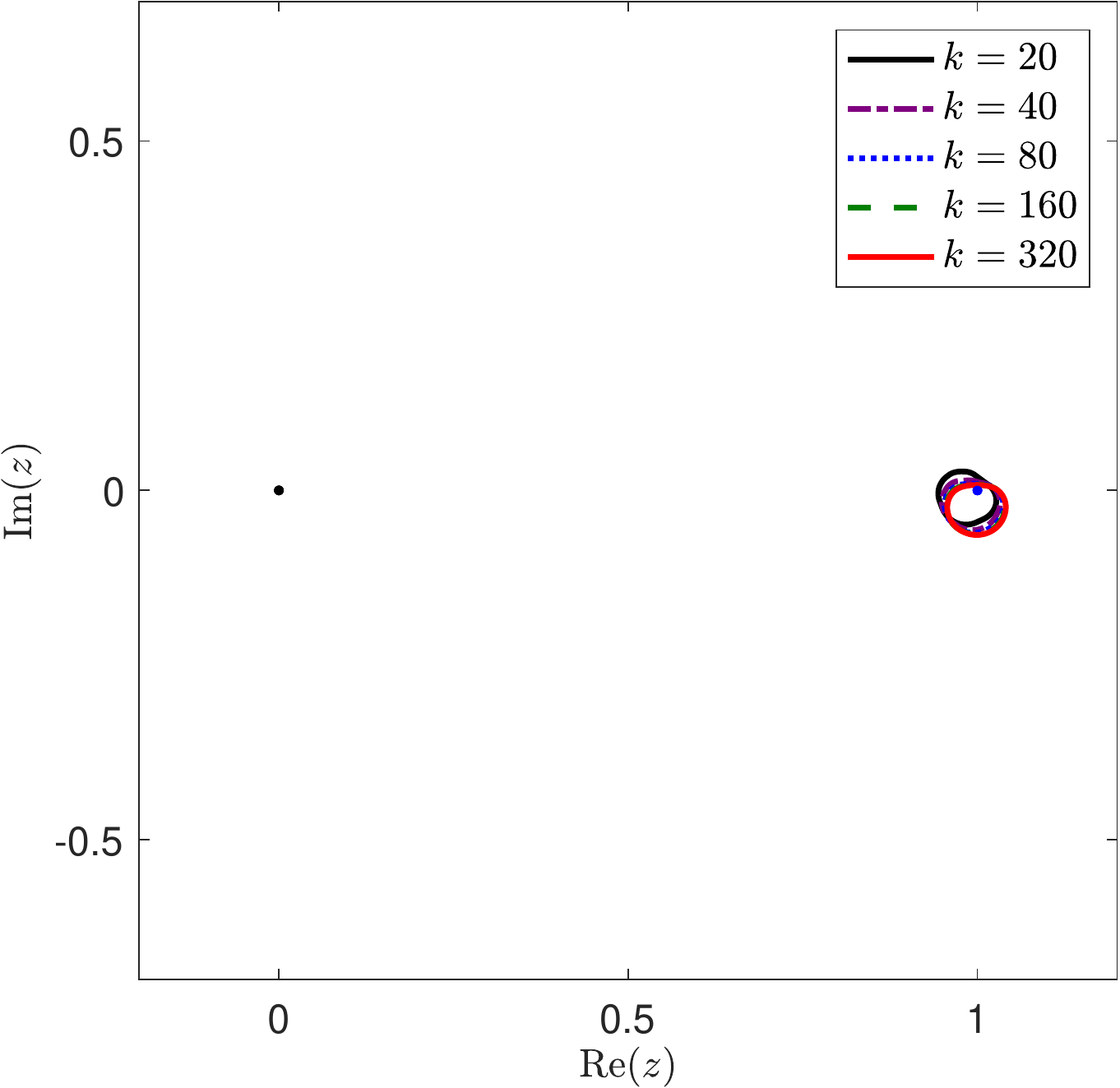}\end{subfigure}
\caption{Field of values of $\A\A_{\eps}^{-1}$ (left) and $\A\B_{\eps}$
(right)
for a 1D Helmholtz problem and various values of $k$, with $\eps = k^2$
and
$N_H \sim k^{3/2}$.}
\label{fig:fov3}
\end{figure}

\begin{figure}[t]
\centering
\begin{subfigure}[b]{0.49\textwidth}
\centering
\includegraphics[width=0.95\textwidth]{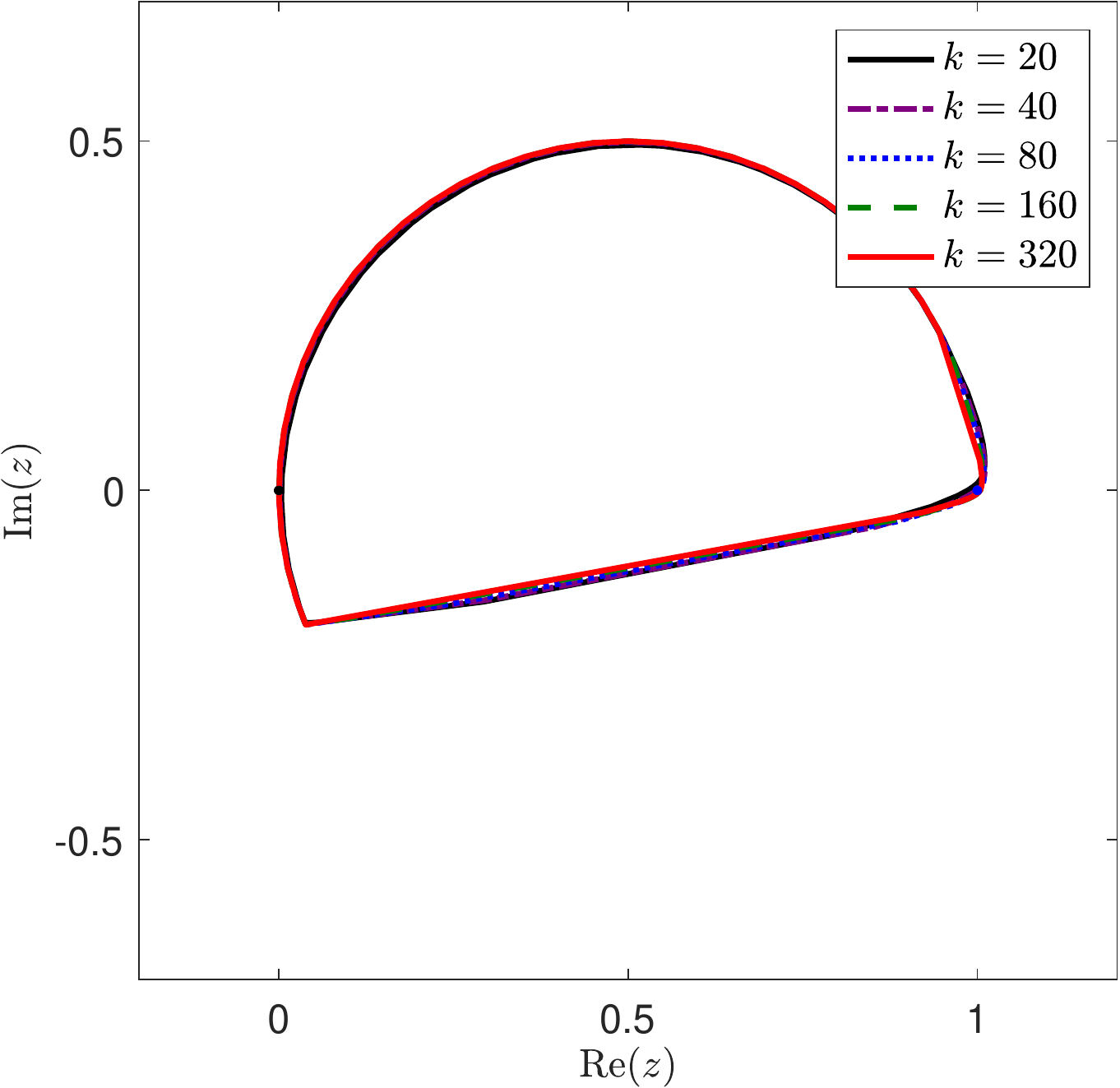}\end{subfigure}
\begin{subfigure}[b]{0.49\textwidth}
\centering
\includegraphics[width=0.95\textwidth]{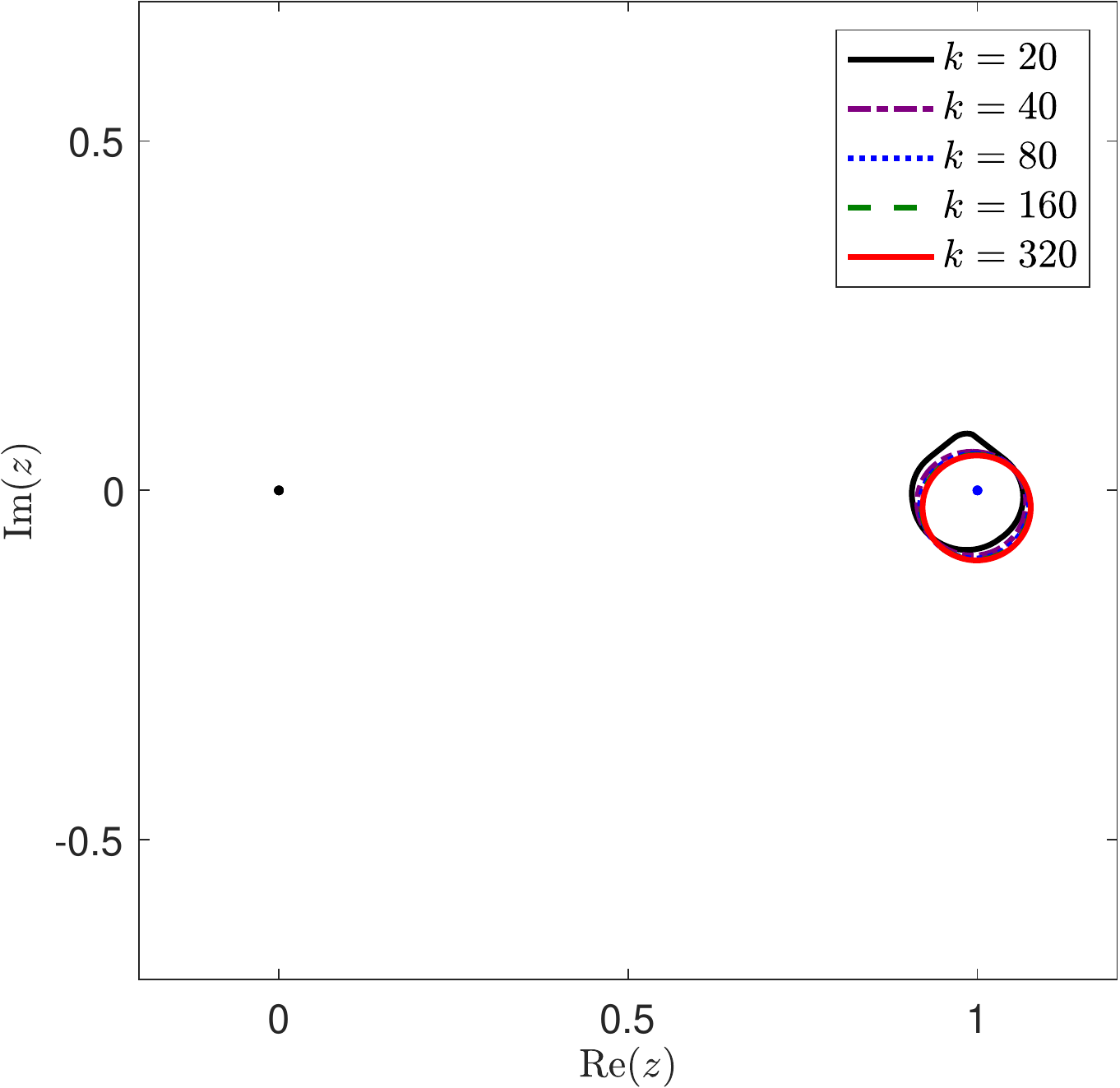}\label{fig-b}
\end{subfigure}
\caption{Field of values of $\A\A_{\eps}^{-1}$ (left) and $\A\B_{\eps}$
(right)
for a 1D Helmholtz problem and various values of $k$, with $\eps = 5k^2$
and
$N_H \sim k^{3/2}$.}
\label{fig:fov4}
\end{figure}

\subsection*{Experiment 2} For our next experiment, we use GMRES to solve the interior impedance problem \eqref{eq:helmholtz} on the unit square $(0,1)^2$, discretized  with  a uniform triangular grid. The right hand side is the
constant vector of ones, and the initial guess the zero vector. We compare the
complex shifted Laplace (CSL) preconditioner  and the
two-level
preconditioner
(TL) for various values of the shift $\eps$ and distinct coarsening levels. The CSL preconditioner is inverted with one multigrid F(1,1) cycle with
$\omega$-Jacobi smoothing on all levels, where $\omega= 0.6$. The grid is chosen as follows:
starting with a coarse grid with $3$ points in each direction, we refine the
grid uniformly until we obtain a number of points (in one dimension) larger than $\lceil \alpha k^{3/2}\rceil$ where $\alpha = 0.6$. The two-level
preconditioner is tested using three different coarse grids.
If $h$ is the fine meshsize in one dimension, the coarse meshsizes for the three
different methods TL-1, TL-2 and TL-3 are $H=2h,4h,8h$ respectively. Since the coarse grid matrices are still large, we use an incomplete LU factorization with drop tolerance of $10^{-6}$ to simulate the exact solve of the coarse grid systems.
The results are shown in table \ref{table:gmres_csl_vs_adef}. Although the meshsize scales with $k^{-3/2}$ (not with $k^{-2}$, as required by the theory), the number of iterations remains constant when the coarse meshes of meshsize $2h$ and $4h$ are used. For the coarse meshsize $8h$ the number of iterations increases linearly, although at a much slower rate than the number of iterations of GMRES preconditioned by the standalone shifted Laplacian. This experiment shows that the theoretical results are not sharp and that wave-number independent convergence can be obtained also with pollution-free meshes where the mesh size scales with $k^{-3/2}$.
Note that increasing the shift $\eps$ leads to an increase in the number of iterations with the standalone CSL preconditioner, and  for the wavenumbers $k=80$ and $k=100$ with the shift $\eps=5k^2$ the GMRES method fails to reach the stopping criterion after 200 iterations. In contrast, the number of iterations remains bounded for the two-level preconditioner even for larger $\eps$.

\begin{table}[t]
\centering
\small
\begin{tabular}{c|cccc|cccc|cccc|}
\cline{2-13}
& \multicolumn{4}{|c|}{$\eps=k^2$}& \multicolumn{4}{c|}{$\eps=2k^2$} & \multicolumn{4}{c|}{$\eps=5k^2$} \\
\hline
 \multicolumn{1}{|c|}{$k$} & CSL & TL-1 & TL-2 & TL-3 & CSL & TL-1 & TL-2 & TL-3 & CSL & TL-1 & TL-2 & TL-3 \\ \hline
\multicolumn{1}{|c|}{$10$} & 14 & 7 & 8 & 10 & 19 & 7 & 9 & 12 & 29 & 7 & 11 & 17 \\
\multicolumn{1}{|c|}{$20$} & 26 & 7 & 10 & 16 & 38 & 7 & 11 & 20 & 65 & 7 & 13 & 29  \\
\multicolumn{1}{|c|}{$40$} & 45 & 6 & 9 & 17 & 79 & 6 & 9 & 19 & 146 & 6 & 9 & 22 \\
\multicolumn{1}{|c|}{$80$} & 78 & 7 & 12 & 36 & 149 & 7 & 12 & 42 & - & 7 & 12 & 46  \\
\multicolumn{1}{|c|}{$100$}& 99 & 6 & 9 & 21 & 184 & 6 & 9 & 22 & - & 6 & 9 & 23\\
\hline
\end{tabular}
\caption{Experiment 2. Number of GMRES iterations for the Helmholtz linear system preconditioned by the  CSL and the two-level method (TL) for various values of the shift $\eps$ and different levels of coarsening.}
\label{table:gmres_csl_vs_adef}
\end{table}

\subsection*{Experiment 3} In our next experiment, we solve again the interior impedance problem \eqref{eq:helmholtz} on the unit square $(0,1)^2$ with the same discretization, initial guess and right hand side as in our previous experiment, but this time we use a multilevel extension of the preconditioner, i.e., a multilevel Krylov method \cite{Erlangga:2008bp, Sheikh:2016eg}. This setting is not included in our theory but is more practically relevant since in realistic applications a two-grid preconditioner can be two expensive to apply. The multilevel preconditioner is implemented within a flexible GMRES (FGMRES) iteration \cite{Saad:1993fc}, and every inexact coarse-grid solve (with a fixed small number of iterations) is performed by another FGMRES iteration, continuing (recursively) through all the grids until the coarsest grid is reached. For more details on the implementation of multilevel Krylov methods we refer the reader to \cite{Erlangga:2008bp,Kehl:2019hj}.

To set up a multilevel Krylov method requires fixing a number of iterations for the intermediate levels. We denote by MK$(l$,$m$,$n$) a method with $l,m,n$ iterations in the second, third and fourth level grid respectively, and one iteration in the remaining coarser grid levels. In this experiment we compare the CSL with the multilevel Krylov methods MK(8,4,2) and MK(6,4,2). The results are shown in Table 
\ref{table:mlgmres_csl_vs_ml}. Similarly as in the two-level case, the multilevel preconditioner outperforms the CSL and requires a constant number of iterations to reach the desired tolerance, even though the intermediate coarse solves are done only inexactly with a small number of iterations. Note also that decreasing the number of iterations in the second level only increases the number of outer iterations by one or two, and the computation times of the two methods MK(8,4,2) and MK(6,4,2) are very similar.

\begin{table}[t]
\centering
\small
\begin{tabular}{c|cc|cc|cc|cc|cc|cc|}
\cline{2-13}
& \multicolumn{6}{|c|}{$\eps = k^2$} & \multicolumn{6}{|c|}{$\eps = 2k^2$}  \\
\cline{1-13}
\multicolumn{1}{|c|}{}& \multicolumn{2}{|c|}{CSL}   & \multicolumn{2}{|c|}{MK(8,4,2)} & \multicolumn{2}{|c|}{MK(6,4,2)} & \multicolumn{2}{|c|}{CSL}   & \multicolumn{2}{|c|}{MK(8,4,2)} & \multicolumn{2}{|c|}{MK(6,4,2)} \\
\multicolumn{1}{|c|}{$k$}& Iter & Time &Iter & Time &Iter &Time & Iter & Time &Iter & Time &Iter &Time \\
\hline
\multicolumn{1}{|c|}{20} & 26 & 0.42 & 7 & 0.77 & 7 & 0.54 & 38  & 0.54   & 7 & 0.70 & 7 & 0.48 \\
\multicolumn{1}{|c|}{40} & 45  & 3.18    & 6  & 3.54     & 6  & 2.39        & 79  & 6.79   & 6      & 2.83     & 6   & 2.14     \\
\multicolumn{1}{|c|}{80} & 78  & 31.78   & 7  & 8.85     & 7  & 6.75        & 149 & 93.73  & 7      & 9.19     & 7   & 7.09     \\
\multicolumn{1}{|c|}{120} & 130 & 377.22  & 7  & 32.812   & 8  & 30.14       & -   & -      & 7      & 33.28  & 8   & 30.02  \\
\multicolumn{1}{|c|}{160} & 142 & 8121.02 & 6  & 124.87   & 8  & 129.24      & -   & -      & 7      & 150.37  & 8   & 127.34 \\
\hline
\end{tabular}
\caption{Experiment 3. Number of GMRES iterations and computation time (in seconds) for the Helmholtz linear system preconditioned by the  CSL and the multilevel Krylov method (MK).}
\label{table:mlgmres_csl_vs_ml}
\end{table}

\subsection*{Experiment 4} In our final experiment we solve the impedance problem on the square $\Omega = (0,1)^2$ with a space-dependent wavenumber. This problem is adapted from \cite{Erlangga:2004bg}. The space-dependent wavenumber is given by
\begin{equation*}	
k(x,y)= \begin{cases}
   (4/3)k_{\text{ref}} & \text{if } 0 \leq y <0.2x+0.2\\
    k_{\text{ref}} & \text{if } 0.2x+0.2 \leq y <-0.2x+0.8\\
    2k_{\text{ref}} & \text{if }-0.2x+0.8\leq y <1
\end{cases}
\end{equation*}
where $k_{\text{ref}} > 0$ is a reference wavenumber. This function is depicted in Figure \ref{fig:wedge-domain}.  As the reference wavenumber $k_{ref}$ is varied, we choose the number of points for a uniform triangular mesh similarly as in the previous problems, with the number of points in one direction proportional to  
$\lceil \alpha k_{ref}^{3/2}\rceil$ with $\alpha = 1.1$. The  larger value of $\alpha$ is chosen to take into account the fact that the maximum wavenumber over the domain is $2k_{ref}$. Similarly to the previous experiments, the CSL preconditioner is compared here with the multilevel Krylov methods MK(8,4,2) and MK(6,4,2). The results are shown in Table \ref{table:mlgmres_csl_vs_ml_wedge}. 
Similarly to the previous experiments, the number of iterations of the CSL preconditioner grows linearly with the wavenumber, and the number of iterations with either of the multilevel Krylov methods remains constant and the computation times are greatly reduced. 

\begin{figure}[t]
\centering
\begin{tikzpicture}[scale = 3]
\draw[thick] (0,0) rectangle (1,1);
\draw[thick] (0, 0.2) -- (1, 0.4);
\draw[thick] (0, 0.8) -- (1, 0.6);
\node[] at (0.5, 0.85) {$2 k_\text{ref}$};
\node[] at (0.5, 0.5) {$k_\text{ref}$};
\node[] at (0.5, 0.15) {$\frac{4}{3} k_\text{ref}$};
\end{tikzpicture}
\caption{Space-dependent wavenumber in model problem 4.}
\label{fig:wedge-domain}
\end{figure}
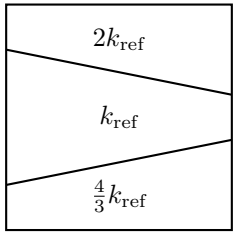

\begin{table}[t]
\centering
\small
\begin{tabular}{c|cc|cc|cc|cc|cc|cc|}
\cline{2-13}
& \multicolumn{6}{|c|}{$\eps = k^2$} & \multicolumn{6}{|c|}{$\eps = 2k^2$}  \\
\cline{1-13}
\multicolumn{1}{|c|}{}& \multicolumn{2}{|c|}{CSL}   & \multicolumn{2}{|c|}{MK(8,4,2)} & \multicolumn{2}{|c|}{MK(6,4,2)} & \multicolumn{2}{|c|}{CSL}   & \multicolumn{2}{|c|}{MK(8,4,2)} & \multicolumn{2}{|c|}{MK(6,4,2)} \\
\multicolumn{1}{|c|}{$k_{ref}$}& Iter & Time &Iter & Time &Iter &Time            & Iter & Time &Iter & Time &Iter &Time \\
\hline
\multicolumn{1}{|c|}{20} & 33 & 1.23 & 6   & 2.02 & 6 & 0.96                        & 52  & 1.44 & 6 & 1.17 & 6 & 0.87 \\
\multicolumn{1}{|c|}{40} & 60 & 24.16 & 6   & 8.01 & 6 & 6.22                   & 105 & 59.96 & 6 & 7.68 & 6 & 5.97   \\
\multicolumn{1}{|c|}{60} & 81 & 132.98 & 5   & 23.42 & 6 & 22.63                & 153 & 424.3 & 5 & 25.53 & 6 & 21.7\\
\multicolumn{1}{|c|}{80} & 106 & 239.8 & 6  & 31.24 & 7 & 34.10                  & 194 & 649.99 & 6 & 28.01 & 7 & 25.27  \\
\multicolumn{1}{|c|}{100} & 126 & 7614.88 & 6 & 137.77 & 6 & 93.90                  & - & - & 6 & 139.88  & 7 & 105.456  \\
\hline
\end{tabular}
\caption{Experiment 4. Number of GMRES iterations and computation time (in seconds) for the Helmholtz linear system with space-dependent wavenumber  preconditioned by the  CSL and the multilevel Krylov method (MK).}
\label{table:mlgmres_csl_vs_ml_wedge}
\end{table}

\section{Conclusions} 
In this paper we have presented a two-level shifted Laplace preconditioner for Helmholtz problems discretised with finite elements. We used the convergence theory of GMRES based on the field of values
to rigorously establish that GMRES will converge in a  number of iterations independent of the wavenumber if a condition of the form $HK^2 \leq C$ holds, for a constant $C$ independent of the wavenumber $k$ but possibly dependent on the size of the complex shift $\eps$. We have also shown in numerical experiments that wavenumber independent convergence can also be obtained under the weaker condition $HK^{3/2} \leq C$, using a multilevel extension of the preconditioner (a multilevel Krylov method with inexact coarse grid solves), and for a test problem with heterogeneous wavenumber.

\bibliographystyle{siam}
\bibliography{biblio.bib}

\end{document}